\documentclass[hidelinks, onefignum,onetabnum]{siamart220329}

\usepackage{graphicx}
\usepackage{color}
\usepackage{amsmath}
\usepackage{amssymb}
\usepackage{xspace}
\usepackage{epsfig}
\usepackage{soul}
\usepackage{xfrac}
\usepackage{apptools} 

\usepackage{amsfonts}
\usepackage{epstopdf}
\ifpdf
  \DeclareGraphicsExtensions{.eps,.pdf,.png,.jpg}
\else
  \DeclareGraphicsExtensions{.eps}
\fi
\usepackage{amsopn}

\headers{Brillouin Zones of Integer Lattices and Their Perturbations}
{Edelsbrunner, Garber, Ghafari, Heiss, Saghafian, and Wintraecken}

\title{Brillouin Zones of Integer Lattices and Their Perturbations\thanks{Submitted to the editors on April 17, 2022. 
\funding{The second author is partially supported by the Alexander von Humboldt Foundation.
  The sixth author is supported by the European Union's Horizon 2020 research and innovation programme under the Marie Sk{\l}odowska-Curie grant agreement no.\ 754411, and by the Austrian Science Fund (FWF), grant no.\ M-3073.
  All other authors are supported by the European Research Council (ERC), grant no.\ 788183, by the Wittgenstein Prize, Austrian Science Fund (FWF), grant no.\ Z 342-N31, and by the DFG Collaborative Research Center TRR 109, Austrian Science Fund (FWF), grant no.\ I 02979-N35.}}}

\author{Herbert Edelsbrunner\thanks{IST Austria (Institute of Science and Technology Austria), Kloster\-neu\-burg, Austria
  (\email{herbert.edelsbrunner@ist.ac.at}, \email{teresa.heiss@ist.ac.at}, \email{morteza.saghafian@ist.ac.at}).}
\and Alexey Garber\thanks{School of Mathematical and Statistical Sciences, University of Texas Rio Grande Valley, Brownsville, TX, USA 
  (\email{alexey.garber@utrgv.edu}).}
\and Mohadese Ghafari\thanks{Khoury College of Computer Sciences, Northeastern University, Boston, MA, USA 
  (\email{ghafari.m@northeastern.edu}).}
\and Teresa Heiss\footnotemark[2]
\and Morteza Saghafian\footnotemark[2]
\and Mathijs Wintraecken\thanks{Inria centre Universit{\'e} C{\^o}te d'Azur, France 
  (\email{mathijs.wintraecken@inria.fr}).}}

\ifpdf
\hypersetup{
  pdftitle={Brillouin Zones of Integer Lattices and Their Perturbations},
  pdfauthor={H. Edelsbrunner, A. Garber, M. Ghafari, T. Heiss, M. Saghafian, and M. Wintraecken}
}
\fi

\AtAppendix{\counterwithin{theorem}{section}} 

\newcommand {\mm}[1] {\ifmmode{#1}\else{\mbox{\(#1\)}}\fi}

\newcommand {\scalprod}[2] {{\langle #1 , #2 \rangle}}

\newsavebox{\smallProofsym}                 
\savebox{\smallProofsym}                               %
{
\begin{picture}(6,6)
\put(0,0){\framebox(6,6){}}
\put(0,2){\framebox(4,4){}}
\end{picture} 
}  

\makeatletter
\long\def\@makecaption#1#2{%
  \vskip\abovecaptionskip
  \sbox\@tempboxa{\small #1: #2}%
  \ifdim \wd\@tempboxa >\hsize
    \small #1: #2\par
  \else
    \global \@minipagefalse
    \hb@xt@\hsize{\hfil\box\@tempboxa\hfil}%
  \fi
  \vskip\belowcaptionskip}
\makeatother

\newcommand{\Rspace}        {\mm{{\mathbb R}}}
\newcommand{\Sspace}        {\mm{{\mathbb S}}}
\newcommand{\Zspace}        {\mm{{\mathbb Z}}}
\newcommand{\Acal}          {\mm{{\mathcal A}}}

\newcommand{\BZone}[3]      {\mm{{\rm Zone}_{#1}{({#2},{#3})}}}
\newcommand{\Brillouin}[2]  {\mm{{\rm Bri}_{#1}{({#2})}}}

\newcommand{\Volume}[1]     {\mm{{\rm vol}{({#1})}}}

\newcommand{\domain}[3]     {\mm{{\rm dom}_{#1}{({#2},{#3})}}}

\newcommand{\ksets}[3]      {\mm{f_{#1}^{(#2)}{(#3)}}}
\newcommand{\card}[1]       {\mm{{\#}{#1}}}

\newcommand{\interior}[1]   {\mm{\rm int\,}{#1}}
\newcommand{\norm}[1]       {\mm{\|{#1}\|}}
\newcommand{\Edist}[2]      {\mm{\|{#1}-{#2}\|}}

\newcommand{\ee}            {\mm{\varepsilon}}

\begin{document}

\maketitle

\begin{abstract}
  For a locally finite set, $A \subseteq \Rspace^d$, the \emph{$k$-th Brillouin zone} of $a \in A$ is the region of points $x \in \Rspace^d$ for which $\Edist{x}{a}$ is the $k$-th smallest among the Euclidean distances between $x$ and the points in $A$.
  If $A$ is a lattice, the $k$-th Brillouin zones of the points in $A$ are translates of each other, which tile space.
  Depending on the value of $k$, they express medium- or long-range order in the set.
  We study fundamental geometric and combinatorial properties of Brillouin zones, focusing on the integer lattice and its perturbations.
  Our results include the stability of a Brillouin zone under perturbations, a linear upper bound on the number of chambers in a zone for lattices in $\Rspace^2$, and the convergence of the maximum volume of a chamber to zero for the integer lattice.
\end{abstract}



\section{Introduction}
\label{sec:1}

Brillouin zones were introduced by L\'{e}on Brillouin \cite{Bri30,Bri46} to describe quantum properties of crystals modeled as lattices in $\Rspace^3$.
Given a locally finite set, $A \subseteq \Rspace^d$, and a specific point, $0 \in A$, we introduce regions of $\Rspace^d$ based on the distances to the points in $A$.
Indeed, the \emph{$k$-th Brillouin zone} of $0$ in $A$ consists of the points $x \in \Rspace^d$ such that at most $k-1$ points in $A \setminus \{0\}$ are closer than $0$ to $x$, and at least $k-1$ points in $A \setminus \{0\}$ are at the same distance or closer than $0$ to $x$.
For $k \geq 2$, it consists of a collection of chambers in the arrangement of bisectors between $0$ and other points in $A$---known as \emph{Bragg planes}---which form a thickened sphere surrounding $0$.
They have been used to analyze the soft density of lattices \cite{EdIg18} and to construct fingerprints of crystal structures modeled as locally finite periodic sets \cite{EHKSW21}.
If $A$ is a lattice, then every Brillouin zone of every point in $A$ has the same ($d$-dimensional) volume, which is equal to the volume of the lattice's unit cell; see \cite{Bie39}.
Among other questions, we probe to what extent this long-range behavior changes when we perturb the lattice.
More generally, we study fundamental geometric and combinatorial questions about Brillouin zones, with an eye on applications to sets with some notion of order, such as lattices, sets with aperiodic structure, and hyperuniform sets.
For background on lattices, we refer to the books by Engel, Michel, Senechal \cite{EMS04} and Zhilinsky \cite{Zhi15} but also to the paper by Skriganov \cite{Skr87}, which focuses on the connection to the geometry of numbers.
For an extensive introduction to aperiodic order see the book by Baake and Grimm \cite{BG13}.
For background on hyperuniform sets see the article by Torquato \cite{Tor18}.
Among our results are
\smallskip \begin{itemize}
  \item bounds on the distance of the $k$-th Brillouin zone of $0$ from $0$;
  \item the stability of the $k$-th Brillouin zone under perturbations of the points;
  \item bounds on the number of chambers in the $k$-th Brillouin zone;
  \item bounds on the maximum diameter of a chamber in the $k$-th Brillouin zone.
\end{itemize} \smallskip
We focus on the integer lattice in $\Rspace^d$ and on its perturbations.
Some of our results hold more generally---such as the stability, which holds for Delone sets---while others are specific---such as the $O(k)$ bound on the number of chambers in the $k$-th Brillouin zone, which we can only prove for lattices in $\Rspace^2$.
We provide experimental data for sets in the plane and use it to formulate concrete questions aimed at deepening the study started in this paper. The corresponding {\tt Python} code is available at \cite{Gha23}.

\smallskip \noindent
\emph{Outline.}
Section~\ref{sec:2} provides the necessary geometric  background.
Sections~\ref{sec:3} and \ref{sec:4} study the width, the distance from the generating point, and the stability of the Brillouin zones.
Section~\ref{sec:5} counts the chambers in the Brillouin zones.
Section~\ref{sec:6} proves bounds on the size of the largest chamber in a Brillouin zone.
Section~\ref{sec:7} concludes the paper.

\section{Geometric Background}
\label{sec:2}

In this section, we introduce the necessary background on Brillouin zones, the related bisector arrangements, and Voronoi tessellations.

\subsection{Types of Sets}
\label{sec:2.1}

The results in this paper apply to a small number of different types of point sets in Euclidean space.
The primary concern is that their Voronoi tessellations are well defined and that the dual Delaunay mosaic covers the entire space.

\smallskip
A set $A \subseteq \Rspace^d$ is \emph{Delone}\footnote{\emph{Delaunay mosaics} and \emph{Delone sets} are both named after Boris Delone (Delaunay), a Russian and Soviet mathematician of French descent.  He used the French spelling Delaunay in earlier works and the transliteration of Russian spelling Delone in later works.} if there are constants $0 < r < R < \infty$ such that every open ball of radius $r$ contains at most one point of $A$, and every closed ball of radius $R$ contains at least one point of $A$.
The supremum $r$ is the \emph{packing radius} and the infimum $R$ is the \emph{covering radius} of $A$.
The existence of $r > 0$ implies that every closed ball contains only a finite number of points in $A$, so the Voronoi tessellation is well defined.
Such a set is called \emph{locally finite}, but note that a locally finite set does not necessarily have a positive packing radius.
The existence of $R < \infty$ implies that every half-space contains infinitely many points of $A$, so the Delaunay mosaic covers $\Rspace^d$.
We therefore call such a set \emph{coarsely dense}, but note that a coarsely dense set does not necessarily have a finite covering radius.

\smallskip
Assuming $d$ linearly independent vectors, $v_1, v_2, \ldots, v_d \in \Rspace^d$, the set of integer combinations, $\Lambda = \{ \sum\nolimits_{i=1}^d j_i v_i \mid j_i \in \Zspace\}$, is a (full rank) \emph{lattice}.
It necessarily contains the origin, denoted $0 \in \Rspace^d$.
If the $v_i$ are the vectors in the standard basis of $\Rspace^d$, we call $\Lambda = \Zspace^d$ the \emph{integer lattice} in $\Rspace^d$.
The $v_i$ span a parallelepiped, and the (absolute) determinant of the $v_i$ is the $d$-dimensional volume of this parallelepiped, which is determined by the lattice.

\smallskip
A \emph{periodic set} is the sum of a lattice and a finite set: $A = \Lambda + M$, in which $M$ is called the \emph{motif}.
Note that every periodic set, and therefore every lattice is Delone.

\subsection{Brillouin Zones}
\label{sec:2.2}

In some areas of mathematics, ``Brillouin zone'' is a synonym for ``Voronoi domain''.
We define them so the zones depend on a positive integer parameter.
\begin{definition}[Brillouin Zones]
  \label{dfn:Brillouin_zones}
  Let $A \subseteq \Rspace^d$ be a locally finite point set with a distinguished point, $0 \in A$, and let $k \geq 1$ be an integer.
  The \emph{$k$-th Brillouin zone} of $0$ is the set $\BZone{k}{0}{A}$ of points $x \in \Rspace^d$ such that
  \begin{align}
    \Edist{x}{a} &< \Edist{x}{0} \mbox{\rm ~for at most~} k-1 \mbox{\rm ~points~} a \in A \setminus \{0\} ,
    \label{eqn:SC1} \\
    \Edist{x}{b} &\leq \Edist{x}{0} \mbox{\rm ~for at least~} k-1 \mbox{\rm ~points~} b \in A \setminus \{0\}.
    \label{eqn:SC2}
  \end{align}
\end{definition}
In the same way we can define $\BZone{k}{a}{A}$ for every point $a \in A$.
Note that $\BZone{1}{0}{A}$ is the Voronoi domain of $0$:  all points $x \in \Rspace^d$ for which no point in $A$ is closer to $x$ than $0$.
To show that Brillouin zones are closed, observe the following. When fixing a set of at least $k-1$ points for \eqref{eqn:SC1} and its subset of at most $k-1$ points for \eqref{eqn:SC2}, the sets defined by \eqref{eqn:SC1} and \eqref{eqn:SC2} are closed and so is their intersection.
The $k$th Brillouin zone of $0$ is the (possibly infinite) union of these closed sets over all suitable choices of sets of at least $k-1$ points and the corresponding subsets. However, its intersection with every compact neighborhood is determined by a finite number of points of $A$ and therefore it is closed as the finite union of closed sets.
Because this holds for any compact neighborhood the $k$th Brillouin zone is closed.
Denoting the open ball with center $x$ and radius $\norm{x} = \Edist{x}{0}$ by $B(x, \norm{x})$, note that \eqref{eqn:SC1} and \eqref{eqn:SC2} 
state there are 
at least $k-1$ points of $A\setminus \{0 \}$ in $B(x, \norm{x})$ and at most $k-1$ in the closure of $B(x, \norm{x})$.
To construct the Brillouin zones, we draw the bisectors defined by $0$ and all points $a \in A \setminus \{0\}$; see Figure~\ref{fig:zones}.
\begin{figure}[hbt]
  \centering
    \vspace{-0.0in}
    \resizebox{!}{1.94in}{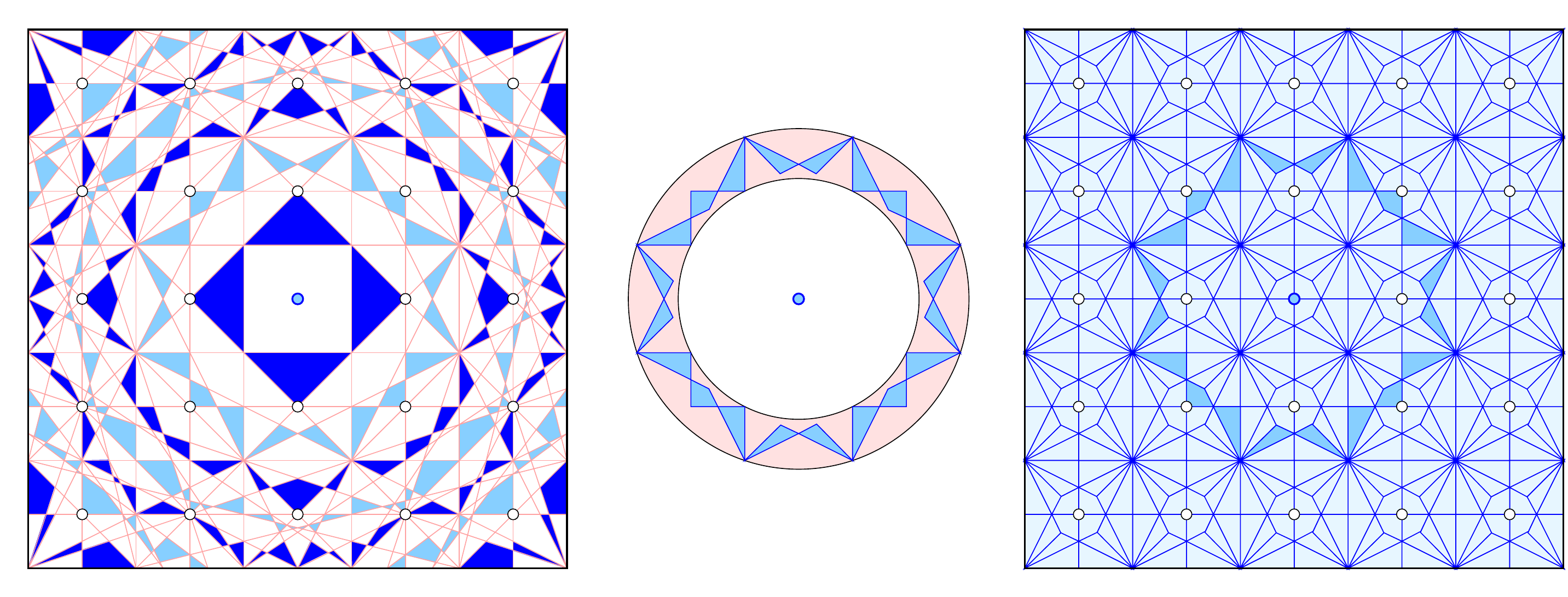} 
    \vspace{-0.1in}
    \caption{\emph{Left:} the arrangement of bisectors defined by the point in the center and all other points in the integer lattice.
    Starting with the second, every fourth Brillouin zone is colored \emph{dark blue} alternating with \emph{light blue}.
    \emph{Middle:} the $6$-th Brillouin zone sandwiched between two circles centered at the point in the center.
    \emph{Right:} the order-$k$ Brillouin tessellation of the integer lattice obtained by overlaying the order-$5$ with the order-$6$ Voronoi tessellations or, equivalently, by drawing the $6$-th Brillouin zones of all points in the integer lattice.}
  \label{fig:zones}
\end{figure}
Observe that $a\in B(x, \norm{x})$ iff the bisector of $0$ and $a$ separates $x$ from $0$, and $b\in \partial B(x, \norm{x})$ iff the bisector of $0$ and $b$ passes through $x$. Hence, according to Definition~\ref{dfn:Brillouin_zones}, $x \in \BZone{k}{0}{A}$ iff at most $k-1$ bisectors separate $x$ and $0$ and at least $k-1$ bisectors pass through $x$ or separate $x$ and $0$.
We can therefore label each $d$-dimensional cell in the arrangement by the number of bisectors that separate points in its interior from $0$, and get $\BZone{k}{0}{A}$ as the union of all (closed) cells labeled $k-1$.
For $k \geq 2$, every Brillouin zone is the difference between the (closed) star-convex set defined by \eqref{eqn:SC1} and the (open) star-convex complement of the set defined by \eqref{eqn:SC2}, in which the interior of the former contains the latter set.
It is therefore not difficult to prove the following fundamental result on the shape of the zones:
\begin{lemma}[Thickened Sphere]
  \label{lem:thickened_sphere}
  Let $A \subseteq \Rspace^d$ be locally finite and coarsely dense, and assume $0 \in A$.
  Then for every finite integer $k \geq 2$, the $k$-th Brillouin zone of $0$ has the homotopy type of $\Sspace^{d-1}$.
\end{lemma}
\begin{proof}
  Every ray $L$ emanating from the origin $0$ has non-empty intersection with the $k$-th Brillouin zone of $0$.
  To show this, consider a point $x$ moving along the ray $L$ together with the (moving) ball $B(x, \norm{x})$ centered at $x$ passing through the origin.
  Since $A$ is coarsely dense, the half space orthogonal to $L$ (and passing through $0$) contains infinitely many points of $A$. Thus, there is a point $x$ on the ray $L$ such that the open ball $B(x, \norm{x})$ contains at least $k$ points in $A\setminus\{0\}$, and thus condition \eqref{eqn:SC1} is not fulfilled, so $x$ is not in the Brillouin zone. 
  However, when moving $x$ on $L$ towards the origin, the first time that the open ball $B(x, \norm{x})$ contains at most $k-1$ points in $A\setminus\{0\}$, the closure of that ball contains at least $k$ points in $A\setminus\{0\}$. Hence, we found a point $x_L$ satisfying both \eqref{eqn:SC1} and \eqref{eqn:SC2}.

  Next, we use star-convexity to construct a deformation retraction: Every point $y$ in the $k$-th Brillouin zone of $0$ gets moved by a straight line to $cy$ with $c= \sup \{c^\prime > 0 \ | \ c^\prime y \in S\}$, where $S$ is the open star-convex set defined by the negation of \eqref{eqn:SC2}.
  The image under this retraction of the points $x_L$ for all the rays $L$, and thus the image of the whole Brillouin zone, is homeomorphic to a sphere.
\end{proof}

\subsection{Bisector Arrangements}
\label{sec:2.3}

The bisectors defined by $0$ and the other points of $A$ form an arrangement of $(d-1)$-planes in $\Rspace^d$; see the left panel in Figure~\ref{fig:zones}.
When restricting ourselves to coarsely dense\footnote{For infinite point sets that are not coarsely dense, the cells in the decomposition might not be polyhedra but are generalized convex polyhedra \cite{Gru07}, because a cell can have infinitely many faces, see \cite[Abbildung 3.4]{Voi08}} locally finite point sets, this is a decomposition of $\Rspace^d$ into convex polyhedra of dimension $p$ from $0$ to $d$.
We refer to the $p$-dimensional polyhedra as \emph{$p$-cells} and to the $d$-cells as \emph{chambers}.
Note that the (relative) interiors of the cells partition $\Rspace^d$, while the cells may intersect in shared boundary pieces, which are again cells in the arrangement.

\smallskip
We say $A \subseteq \Rspace^d$ is in \emph{general position} if no $d+1$ of its points lie on a common $(d-1)$-plane and no $d+2$ of its points lie on a common $(d-1)$-sphere.
If $A$ is in general position, the arrangement of bisectors is \emph{simple}; that is: any $d$ of the $(d-1)$-planes meet in a common point and no $d+1$ of them do.
Under this assumption, the number of cells in the arrangement is a function of the number of bisectors; see for example \cite[Theorem 1.3]{Ede87}.
\begin{proposition}[Plane Arrangements]
  \label{prop:plane_arrangements}
  Let $\cal A$ be a simple arrangement of $n$ $(d-1)$-planes in $\Rspace^d$.
  Then for each $0 \leq p \leq d$, the number of $p$-cells in $\cal A$ is $\sum\nolimits_{i=0}^p \tbinom{d-i}{p-i} \tbinom{n}{d-i}$.
\end{proposition}
Assuming $p$ and $d$ are constants, this implies that the arrangement has $\Theta (n^d)$ cells of any dimension.

\subsection{Voronoi Tessellations}
\label{sec:2.4}

For fixed $k \geq 1$, the $k$-th Brillouin zones of the points in $A$ form the \emph{order-$k$ Brillouin tessellation}, which is also known as the \emph{degree-$k$ Voronoi tessellation} \cite{EdSe86}.
It relates to the better known order-$k$ Voronoi tessellation, see e.g.\ \cite{Lee82}, which we introduce first.

\smallskip
Given a finite set $A \subseteq \Rspace^d$ and a non-empty subset $Q \subseteq A$, we define the \emph{region} of $Q$ as the points $x \in \Rspace^d$ that satisfy $\Edist{x}{q} \leq \Edist{x}{a}$ for all $q \in Q$ and $a \in A \setminus Q$.
If non-empty, this region is a convex polyhedron, and if $A$ is in general position, then this polyhedron is $d$-dimensional.
The \emph{order-$k$ Voronoi tessellation} is the polyhedral complex consisting of the regions of subsets $Q \subseteq A$ of size $\card{Q} = k$ and the faces shared by these polyhedra.

\smallskip
For finite sets $A \subseteq \Rspace^2$, the number of cells in the order-$k$ Voronoi tessellation is well understood.
Part of the reason is that in the $2$-dimensional generic case, every vertex is shared by exactly two tessellations of consecutive order.
We thus distinguish between the \emph{old} and \emph{new} vertices of the order-$k$ Voronoi tessellation, which it shares with the order-$(k-1)$ and order-$(k+1)$ Voronoi tessellations, respectively.
Using induction, Lee proved that there are fewer than $[4k-2]n$ vertices, $[6k-3]n$ edges, and $[2k-1]n$ regions \cite{Lee82}.
We will need more precise estimates, so we follow \cite{BCES21} and view the order-$k$ Voronoi tessellation in $\Rspace^2$ as the projection of cells in an arrangement of planes in $\Rspace^3$, and then overcount by moving to the $3$-sphere.
The latter amounts to mapping the planes to $2$-dimensional great-spheres in $\Sspace^3$, which effectively combines the order-$k$ with the order-$(n-k)$ Voronoi tessellation.
We call the result of this view the \emph{spherical order-$k$ Voronoi tessellation.}
The benefit of this approach is that we get equalities for the number of cells, rather than inequalities.
Indeed, applying the $3$-dimensional methods in \cite{BCES21} to $2$ dimensions, things simplify considerably and it is not difficult to count the faces assuming general position:
\begin{proposition}[Spherical Order-$k$ Voronoi Tessellation]
  \label{prop:order-k_Voronoi_tessellation}
  Let $A$ be $n \geq 4$ points in general position in $\Rspace^2$.
  Then for $1 \leq k \leq n-1$, the spherical order-$k$ Voronoi tessellation of $A$ has $u_k = 2(k-1)(n-k)$ old vertices, $w_k = 2k(n-k) - 2k$ new vertices, $e_k = (6k-3)(n-k) - 3k$ edges, and $r_k = (2k-1)(n-k) - (k-2)$ regions.
\end{proposition}

In dimensions three and higher, counting the cells in the order-$k$ Voronoi tessellations is significantly more difficult \cite{BCES21} and only rough upper and lower bounds are known.
The \mbox{order-$k$} Brillouin tessellation is the overlay of the order-$(k-1)$ and order-$k$ Voronoi tessellations; see Figure~\ref{fig:zones}, where the order-$5$ and order-$6$ Voronoi tessellations are overlaid to give the order-$6$ Brillouin tessellation.
It decomposes $\Rspace^d$ into convex regions such that any two points in the same region have the same $k$-th nearest point in $A$.
Each such region is a chamber of the $k$-th Brillouin zone of this $k$-th nearest point in $A$.
In the $2$-dimensional case, it is not difficult to get good bounds on the number of cells from Proposition~\ref{prop:order-k_Voronoi_tessellation}:
\begin{corollary}[Spherical Order-$k$ Brillouin Tessellation]
  \label{cor:order-k_Brillouin_tessellation}
  Let $A$ be $n \geq 4$ points in general position in $\Rspace^2$.
  Then for $1 \leq k \leq n-1$, the spherical order-$k$ Brillouin tessellation of $A$ has $(6k-6)(n-k)-4$ vertices, $(12k-12)(n-k)-6$ edges, and $(6k-6)(n-k)$ regions.
\end{corollary}
\begin{proof}
  For the vertices, we add the numbers of the order-$(k-1)$ and order-$k$ Voronoi tessellations and remove duplicates: $u_{k-1} + u_k + w_k = (6k-6)(n-k) - 4$.
  For the edges, we add the numbers: $e_{k-1} + e_k = (12k-12)(n-k) - 6$.
  To count the regions, we use the Euler formula for the $2$-sphere, which implies $(e_{k-1}+e_k) - (u_{k-1}+u_k+w_k) + 2 = (6k-6)(n-k)$ regions, as claimed.
\end{proof}

If the points are not in general position, then the equations turn into upper bounds, which also hold if we abandon the spherical view and count in the Euclidean plane.

\subsection{Perturbations of the Integer Lattice}
\label{sec:2.5}

We call an injective map, \linebreak 
$\varphi \colon \Zspace^d \to \Rspace^d$, a \emph{perturbation} of the integer lattice, and the supremum of the $\Edist{a}{\varphi (a)}$ over all $a \in \Zspace^d$ its \emph{magnitude}.
The perturbation is \emph{bounded} if its magnitude is finite.
We call the image of the map, $P = \varphi (\Zspace^d )$, a \emph{perturbed integer lattice}.
Without loss of generality, we assume throughout this paper that $0 \in \Zspace^d$ is a fixed point; that is: $\varphi (0) = 0$.
We generate perturbations randomly, by picking $\varphi (a)$ uniformly at random in $a + [-\tau, \tau]^d$ for each $a \in \Zspace^d \setminus \{0\}$, in which $\tau$ is the \emph{strength} of the perturbation.
Standardizing to three strengths, we call the generated perturbation \emph{weak}, \emph{medium}, \emph{strong} if $\tau = 0.02, 0.10, 0.50$, respectively.

\smallskip
While magnitude and strength are different concepts, we discuss them in disjoint contexts and thus use the same letter, $\tau$, to denote either.
The strength of a perturbation is relevant in many of our computational experiments, and in Appendix~\ref{app:A}, which supports the experiments by analyzing how many Brillouin zones in a finite arrangement of bisectors are reliable.

\subsection{Balls and Spheres}
\label{sec:2.6}

Consistent with the common notation in stochastic geometry, we write $\nu_d$ for the $d$-dimensional volume of the unit ball in $\Rspace^d$, and $\sigma_d$ for the $(d-1)$-dimensional volume of the unit sphere, which bounds this ball.
We have
\begin{align}
  \sigma_d  &=  \frac{2 \pi^{d/2}}{\Gamma(\sfrac{d}{2})}
                \renewcommand{\arraystretch}{1.7}
             =  \left\{ \begin{array}{cl}
                  \frac{[2 \pi]^{{d}/{2}}}{[d-2]!!}  &  \mbox{\rm for even~} d , \\
                  \frac{2 [2 \pi]^{{[d-1]}/{2}}}{[d-2]!!}  &  \mbox{for odd~} d ,
                \end{array} \right.
\end{align}
in which $[d-2]!!$ is the product of every other integer starting with $d-2$; see e.g.\ \cite[page 13]{ScWe08}, where $\sigma_d$ is denoted $\omega_d$.
Furthermore, the volume of the unit ball is $\nu_d = {\sigma_d}/{d}$, which is denoted $\kappa_d$ in \cite{ScWe08}.
Writing $B(0, \rho)$ for the ball with radius $\rho$ centered at $0 \in \Rspace^d$, we note that its $d$-dimensional volume is $\Volume{B(0,\rho)} = \nu_d \rho^d$.

\section{Distance and Width}
\label{sec:3}

Drawing the two circles centered at $0 \in A$ whose radii are the minimum and maximum distances of the $k$-th Brillouin zone from $0$, we get an annulus containing $\BZone{k}{0}{A}$; see the middle panel in Figure~\ref{fig:zones}.
We write $r_k(0) < R_k(0)$ for the two distances and call $W_k(0) = R_k(0) - r_k(0)$ the \emph{width} of the zone.
Since the distances and widths are the same for all points of a lattice, we simplify notation to $r_k, R_k, W_k$ whenever we talk about lattices.

\subsection{Integer Lattices} 
\label{sec:3.1}

We give upper and lower bounds on $r_k$ and $R_k$.
Some of these bounds were known to Jones \cite[Section 5]{Jones03}.
In particular, the lower bound on $r_k$ in \cite{Jones03} is the same as in the theorem below, while the upper bound on $r_k$ is slightly weaker.
Similarly, \cite{Jones03} contains a lower bound on the width of the $k$-th Brillouin zone that is weaker than the lower bound in the theorem below.
For completeness, we include a proof of the statement based on the fact that all Brillouin zones of a point in a lattice have the same volume.
\begin{theorem}[Width for Integer Lattices]
  \label{thm:width_for_integer_lattices}
  For every $k \geq 1$, the minimum and maximum distances and the width of the $k$-th Brillouin zone of any point in $\Zspace^d$ satisfy
  \begin{align}
    \sqrt[d]{\sfrac{k}{\nu_d}} - \sfrac{\sqrt{d}}{2}  &<  r_k  <  \sqrt[d]{\sfrac{(k-1)}{\nu_d}} , 
    \label{eqn:mindist-integer-lattice} \\
    \sqrt[d]{\sfrac{k}{\nu_d}}  &<  R_k  <  \sqrt[d]{\sfrac{k}{\nu_d}} + \sfrac{\sqrt{d}}{2} , 
    \label{eqn:maxdist-integer-lattice} \\
    \sqrt[d]{\sfrac{k}{\nu_d}} - \sqrt[d]{\sfrac{(k-1)}{\nu_d}} &<  W_k < \sqrt{d} .
      \label{eqn:width-integer-lattice}
  \end{align}
\end{theorem}
\begin{proof}
  To prove the upper bound for $r_k$ and the lower bounds for $R_k$ and $W_k$, recall that the $1$-st Brillouin zone is $[- \sfrac{1}{2}, \sfrac{1}{2}]^d$.
  Its $d$-dimensional volume and---by a classic result \cite{Bie39}---the volume of any other Brillouin zone is $\Volume{\BZone{1}{0}{\Zspace^d}} = 1$.
  Let $\varrho_k = \sqrt[d]{\sfrac{k}{\nu_d}}$ and note that $\Volume{B(0, \varrho_k)} = k$.
  The first $k-1$ Brillouin zones of $0$ have a total volume of $k-1$, but since their union is not a perfect geometric ball, they do not cover all of $B(0, \varrho_{k-1})$.
  Hence, $\BZone{k}{0}{\Zspace^d}$ contains points in the interior of $B(0, \varrho_{k-1})$, which implies $r_k < \varrho_{k-1}$.
  A symmetric argument implies $R_k > \varrho_k$.
  This implies that the width satisfies
  \begin{align}
    W_k  &=  R_k - r_k  >  \varrho_k - \varrho_{k-1}
      =  \sqrt[d]{\sfrac{k}{\nu_d}} - \sqrt[d]{\sfrac{(k-1)}{\nu_d}} .
  \end{align}
  To prove the lower bound for $r_k$ and the upper bounds for $R_k$ and $W_k$, we use a straightforward solution to the generalization of the Gauss circle problem to $d$ dimensions; see \cite{Hux03} and references therein for the problem and relevant progress.
  Letting $x \in \Rspace^d$, not necessarily in $\Zspace^d$,
  the number of integer points in $B(x,\rho)$ satisfies
  \begin{align}
    \nu_d [\rho-\sfrac{\sqrt{d}}{2}]^d  &< \card{\mbox{\rm points}}
      <  \nu_d [\rho + \sfrac{\sqrt{d}}{2}]^d .
    \label{eqn:points}
  \end{align}
  To see this, we note that the unit cubes of the points inside $B(x,\rho)$ cover $B(x,\rho-\sfrac{\sqrt{d}}{2})$, which implies the first inequality in \eqref{eqn:points} by a volume argument.
  Symmetrically, these unit cubes are contained in $B(x, \rho+\sfrac{\sqrt{d}}{2})$, which implies the second inequality, again by a volume argument.
  Setting $\rho = \norm{x}$ and $\card{\mbox{\rm points}} = k$, we rearrange the two inequalities and get bounds on the minimum and maximum distances from $0$:
  \begin{align}
    \sqrt[d]{\sfrac{k}{\nu_d}} - \sfrac{\sqrt{d}}{2}  &<  r_k  \leq  \norm{x}
      \leq  R_k  <  \sqrt[d]{\sfrac{k}{\nu_d}} + \sfrac{\sqrt{d}}{2} .
    \label{eqn:points2}
  \end{align}
  This implies that the width is $W_k = R_k - r_k  < \sqrt{d}$, as claimed.
\end{proof}

We remark that the lower bound on the width in Theorem~\ref{thm:width_for_integer_lattices} tends to $0$ when $k$ goes to infinity, while the upper bound is a constant independent of $k$.
Using earlier work by van der Corput \cite{Corput20}, Kwakkel improves upon our upper bound in two dimensions, showing that the width of the $k$-th Brillouin zone goes to zero as $k$ goes to infinity \cite[Theorem~3.2]{kwakkel06}.
With appropriate changes of the constants, all bounds in Theorem~\ref{thm:width_for_integer_lattices} extend to Delone sets in $\Rspace^d$.
Indeed, we can use volume arguments to adjust \eqref{eqn:points} to the more general case of Delone sets, while leaving the rest of the argument as is.

\subsection{Perturbed Integer Lattices}
\label{sec:3.2}

We generalize the lower bound for $r_k$ and the upper bound for $R_k$ in Theorem~\ref{thm:width_for_integer_lattices} to perturbations of $\Zspace^d$.
An upper bound for the width follows.
\begin{theorem}[Width for Perturbed Integer Lattices]
  \label{thm:width_for_perturbed_integer_lattices}
  Let $\varphi \colon \Zspace^d \to \Rspace^d$ be a bounded perturbation with magnitude $\tau < \infty$, and let $k \geq 1$.
  Then the distances and the width of the $k$-th Brillouin zone of $0 \in P = \varphi(\Zspace^d)$ satisfy
  \begin{align}
    \label{eqn:perturbed_lower_and_upper_bound}
    \sqrt[d]{\sfrac{k}{\nu_d}} - \sfrac{\sqrt{d}}{2} - \tau  <  r_k(0)  <  R_k(0)  &<  \sqrt[d]{\sfrac{k}{\nu_d}} + \sfrac{\sqrt{d}}{2} + \tau , \\
    W_k(0) &< \sqrt{d} + 2 \tau.
      \label{eqn:width-perturbed}
  \end{align}
\end{theorem}
\begin{proof}
  Write $\varrho_k = \sqrt[d]{\sfrac{k}{\nu_d}}$ and recall from the proof of Theorem~\ref{thm:width_for_integer_lattices} that the closed balls with radii $\varrho_k \pm \sfrac{\sqrt{d}}{2}$ contain at most and at least $k$ points of $\Zspace^d$, respectively.
  Since $\Edist{a}{\varphi (a)} \leq \tau$, for every $a \in P$, the balls with radii $\varrho_k \pm (\sfrac{\sqrt{d}}{2} + \tau)$ contain at most and at least $k$ points of $P$.
  This implies the claimed lower bound for $r_k(0)$ and the claimed upper bound for $R_k(0)$.
  We get \eqref{eqn:width-perturbed} from $W_k(0) = R_k(0) - r_k(0)$.
\end{proof}

\subsection{Distances and Widths Experimentally}
\label{sec:3.3}

We illustrate Theorems~\ref{thm:width_for_integer_lattices} and \ref{thm:width_for_perturbed_integer_lattices} by constructing Brillouin zones in the plane.
The solid graphs in Figure~\ref{fig:width} give the minimum and maximum distances of the Brillouin zones of $0 \in \Zspace^2$ from $0$, which are bracketed by the upper and lower bounds proved in Theorem~\ref{thm:width_for_integer_lattices}.
The width is the difference between these two distances.
The dotted curves in Figure~\ref{fig:width} show the minimum and maximum distances of the first $57$ Brillouin zones of $0 \in P$ from $0$, in which $P$ is a strong perturbation of $\Zspace^2$.
We see that the perturbation causes only minor displacements of the four graphs.
\begin{figure}[hbt]
  \centering
    \vspace{-0.0in}
     \includegraphics[width=0.625\textwidth]{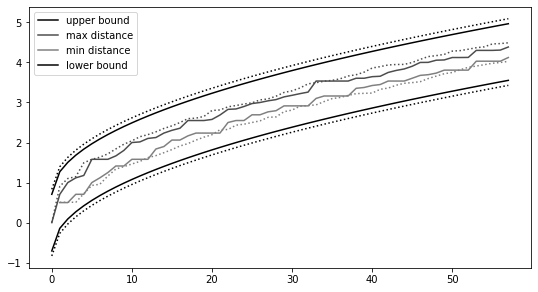}
    \vspace{-0.1in}
    \caption{The \emph{solid} curves show the min and max distances, $r_k$ and $R_k$, of the $k$-th Brillouin zones of $0 \in \Zspace^2$ from $0$, together with their lower and upper bounds.
    For comparison, the \emph{dotted} curves show the same information for a strong perturbation of $\Zspace^2$, so the lowest and highest curves are the bounds from \eqref{eqn:perturbed_lower_and_upper_bound}.
    As detailed in Appendix~\ref{app:A}, the \emph{dotted} curves graphing the min and max distances of the $k$-th Brillouin zone in the perturbed integer lattice are provably correct up to $k = 34$ and may possibly be contaminated by missing bisectors starting from $k = 35$ onward.}
  \label{fig:width}
\end{figure}

\section{Stability}
\label{sec:4}

While Theorem~\ref{thm:width_for_perturbed_integer_lattices} bounds the width under perturbations, it falls short of showing that the Brillouin zones are stable, which we prove in this section.
A related result is the stability of Voronoi regions, which was proven by Reem \cite{Ree11}.

\subsection{Two Technical Lemmas}
\label{sec:4.1}

We begin with an observation about the function $\min_k \colon \Rspace^n \to \Rspace$ that maps a vector of $n$ numbers to the $k$-th smallest among them.
We write $x_1, x_2, \ldots, x_n$ for the components of $x \in \Rspace^n$ and $y_1, y_2, \ldots, y_n$ for the components of $y \in \Rspace^n$.
\begin{lemma}[Stability of Rank]
  \label{lem:stability_of_rank}
  For any $1 \leq k \leq n$, the map $\min_k \colon \Rspace^n \to \Rspace$ is $1$-Lipschitz under the $L_\infty$-norm; that is: if $|x_i - y_i| \leq \ee$ for all $1 \leq i \leq n$, then $| \min_k(x) - \min_k(y)| \leq \ee$.
\end{lemma}
\begin{proof}
  Let $X_0$ contain the $k$ smallest components of $x$, and let $X_1$ contain the $n-k+1$ largest components of $x$, in which ties are broken arbitrarily.
  We note that $\max (X_0) = \min_k (x) = \min (X_1)$.
  Let $y_i \in Y_0$ iff $x_i \in X_0$, and similarly for $Y_1$ and $X_1$.
  Since corresponding components differ by at most $\ee$, we have $\max (Y_0) - \ee \leq \min_k (x) \leq \min (Y_1) + \ee$, and because $\card{Y_0} = k$ and $\card{Y_1} = n-k+1$, we have $\min (Y_1) \leq \min_k (y) \leq \max (Y_0)$.
  Hence, $\min_k(y) - \ee \leq \min_k (x) \leq \min_k (y) + \ee$, which is equivalent to the claimed inequality.
\end{proof}

The next technical lemma asserts the stability of the intersection of a half-line with the bisector of two points.
In the setting we consider, the half-line emanates from one of the two points, which we assume is $0 \in \Rspace^d$.
Let $a \in \Rspace^d \setminus \{0\}$ be the second point, let $u \in \Sspace^{d-1}$ be the direction of the half-line, and assume $\scalprod{u}{a} > 0$ so that the intersection between the half-line and the bisector of $0, a$ exists.
Writing $\lambda u$ for the points of the half-line, we solve $\lambda = \Edist{\lambda u}{a}$ to get $\lambda (a) = \frac{1}{2} \norm{a}^2 / \scalprod{u}{a}$ as the parameter value of the intersection point.
\begin{lemma}[Stability of Crossing]
  \label{lem:stability_of_crossing}
  Let $u \in \Sspace^{d-1}$ be a direction and $a \in \Rspace^d \setminus \{0\}$ a point with $\scalprod{u}{a} > 0$.
  Then for every $\ee > 0$ there exists $\tau > 0$ such that $\Edist{p}{a} < \tau$ implies that $\lambda (p)$ is well defined and satisfies $\lambda (p) < \lambda (a) + \ee$.
\end{lemma}
\begin{proof}
  For $\tau < \scalprod{u}{a}$, we have $\scalprod{u}{p} \geq \scalprod{u}{a - \tau u} = \scalprod{u}{a} - \tau > 0$, so the bisector of $0, p$ has a unique intersection point with the half-line of points $\lambda u$, and this intersection point is given by $\lambda (p) = \frac{1}{2} \norm{p}^2 / \scalprod{u}{p}$.
  Hence,
  \begin{align}
    \frac{\lambda(p)}{\lambda(a)}
    &=  \frac{\norm{p}^2 \scalprod{u}{a}}{\norm{a}^2 \scalprod{u}{p}}
    <  \frac{\norm{a + \frac{\tau a}{\norm{a}}}^2 \scalprod{u}{a}}
               {\norm{a}^2 \scalprod{u}{a - \tau u}}
     =  \frac{\norm{a}^2 \scalprod{u}{a}}
             {\norm{a}^2 \scalprod{u}{a}}
        \frac{(1 + \frac{ \tau}{\norm{a}})^2}
             {(1 - \frac{\tau}{\scalprod{u}{a}})},
    \label{eqn:crossing}
  \end{align}
  in which the first ratio of the right-hand side in \eqref{eqn:crossing} cancels.
  Since $\norm{a}$, $\norm{a}^2$, $\scalprod{u}{a}$, and $\lambda(a)$ are all fixed and positive, it is easy to find a sufficiently small $\tau > 0$ so that the remaining ratio is at most $1 + \ee/\lambda(a)$.
  This implies $\lambda (p) < \lambda (a) + \ee$, as required.
\end{proof}

\subsection{The Stability of Brillouin Zones}
\label{sec:4.2}

Write $\domain{k}{0}{\Zspace^d}$ for the set of points $x \in \Rspace^d$ for which fewer than $k$ points in $\Zspace^d \setminus \{0\}$ have distance less than $\norm{x}$ from $x$, and note that $\domain{k}{0}{\Zspace^d}$ is the union of the first $k$ Brillouin zones of $0$.
Since the Brillouin zones have disjoint interiors, this implies that $\BZone{k}{0}{\Zspace^d}$ is the closure of $\domain{k}{0}{\Zspace^d} \setminus \domain{k-1}{0}{\Zspace^d}$.
For the same reason,
\begin{align}
  \partial \BZone{k}{0}{\Zspace^d} &= \partial \domain{k}{0}{\Zspace^d} \cup \partial \domain{k-1}{0}{\Zspace^d}
\end{align}
for the boundary of the $k$-th Brillouin zone.
For any direction $u \in \Sspace^{d-1}$, consider the half-line of points $\lambda u$, with $\lambda \geq 0$, and write $\alpha_k (u)$ for the unique $\lambda$ such that $\lambda u \in \partial \domain{k}{0}{\Zspace^d}$.
Similarly, write $\beta_k (u)$ for the unique $\lambda$ such that $\lambda u \in \partial \domain{k}{0}{P}$, in which
$\varphi \colon \Zspace^d \to \Rspace^d$ is a perturbation of the integer lattice and $P = \varphi (\Zspace^d)$ is the perturbed set.
\begin{theorem}[Stability of Brillouin Zones]
  \label{thm:stability_of_Brillouin_zones}
  For every integer $k \geq 1$ and real $\ee > 0$, there is a sufficiently small $\tau > 0$ such that for every perturbation $\varphi \colon \Zspace^d \to \Rspace^d$ with $\varphi (0) = 0$ and $\sup \{ \Edist{a}{\varphi (a)} \} < \tau$, we have $|\beta_k(u) - \alpha_k(u)| < \ee$ for every direction $u \in \Sspace^{d-1}$.
\end{theorem}
\begin{proof}
  Assume a sufficiently small $\tau > 0$, and let $\varphi \colon \Zspace^d \to \Rspace^d$ be a perturbation of the integer grid such that $\varphi (0) = 0$ and $\sup \{ \Edist{a}{\varphi (a)} \} < \tau$.
  Fixing a direction $u \in \Sspace^{d-1}$, we begin by constructing a set $A \subseteq \Zspace^d \setminus \{0\}$ that satisfies the conditions needed to apply Lemmas~\ref{lem:stability_of_rank} and \ref{lem:stability_of_crossing}; that is:
  \smallskip \begin{enumerate}
    \item $\scalprod{u}{a} > 0$ so $\lambda(a) = \frac{1}{2} \norm{a}^2 / \scalprod{u}{a}$ is well defined for every $a \in A$,
    \item $\alpha_k (u)$ is the $k$-th smallest of the $\lambda (a)$, $a \in A$,
    \item $\scalprod{u}{p} > 0$ so $\lambda(p) = \frac{1}{2} \norm{p}^2 / \scalprod{u}{p}$ is well defined for every $p \in \varphi (A)$,
    \item $\beta_k (u)$ is the $k$-th smallest of the $\lambda (p)$, $p \in \varphi (A)$.
  \end{enumerate} \smallskip
  In short, to satisfy Conditions 1 and 3, we pick points in the open half-space defined by $\scalprod{x}{u} > 0$, and to satisfy Conditions 2 and 4, we include sufficiently many points whose bisectors intersect the half-line defined by $u$ near $0$.
  To be specific, we set $\lambda_0 = \sqrt[d]{{k}/{\nu_d}} + {\sqrt{d}}/{2} +1$,
  and let $B_0 = B(\lambda_0 u, \lambda_0)$ be the open ball passing through $0$.
  By Theorem~\ref{thm:width_for_integer_lattices}, $B_0$ contains at least $k$ integer points, including the $k$ points whose bisectors with $0$ intersect the half-line at the first $k$ crossings.
  Assuming $\tau \leq 1$, by Theorem~\ref{thm:width_for_perturbed_integer_lattices}, $B_0$ contains at least $k$ points of $\varphi (\Zspace^d)$, including the $k$ points whose bisectors with $0$ intersect the half-line at the first $k$ crossings.
  The set $A$ consists of all points $a \in \Zspace^d \setminus \{0\}$ such that $a\in B_0$ 
  or $\varphi (a) \in B_0$. 
  Observe that for this choice of $A$, Conditions~2 and 4 are satisfied.

  \smallskip
  To establish the remaining two properties,
  note that $a\in B_0 \cap \Zspace^d \setminus \{ 0 \} \subseteq B_0 \setminus B(0,1)$ trivially satisfies $\scalprod{u}{a} > 0 $. 
  Assuming $\tau$ is smaller than $d_1=\inf \{ d(x,H)\ | \ x\in B_0 \setminus B(0,1)\}$, in which $H$ is the half-space of points $y$ that satisfy $\scalprod{u}{y} \leq 0 $,
  the perturbation $p=\varphi(a)$ satisfies $\scalprod{u}{p} > 0 $ as well; see Figure~\ref{fig:proof4dot3}. 
  Note that this condition on $\tau$ depends neither on $u$ nor on $\varphi$, only on the radius $\lambda_0$ and as such on $d$ and $k$.
  \begin{figure}[hbt]
  \centering
    \includegraphics[width=0.8\linewidth, trim = {0cm 0.8cm 0cm 0cm}, clip]{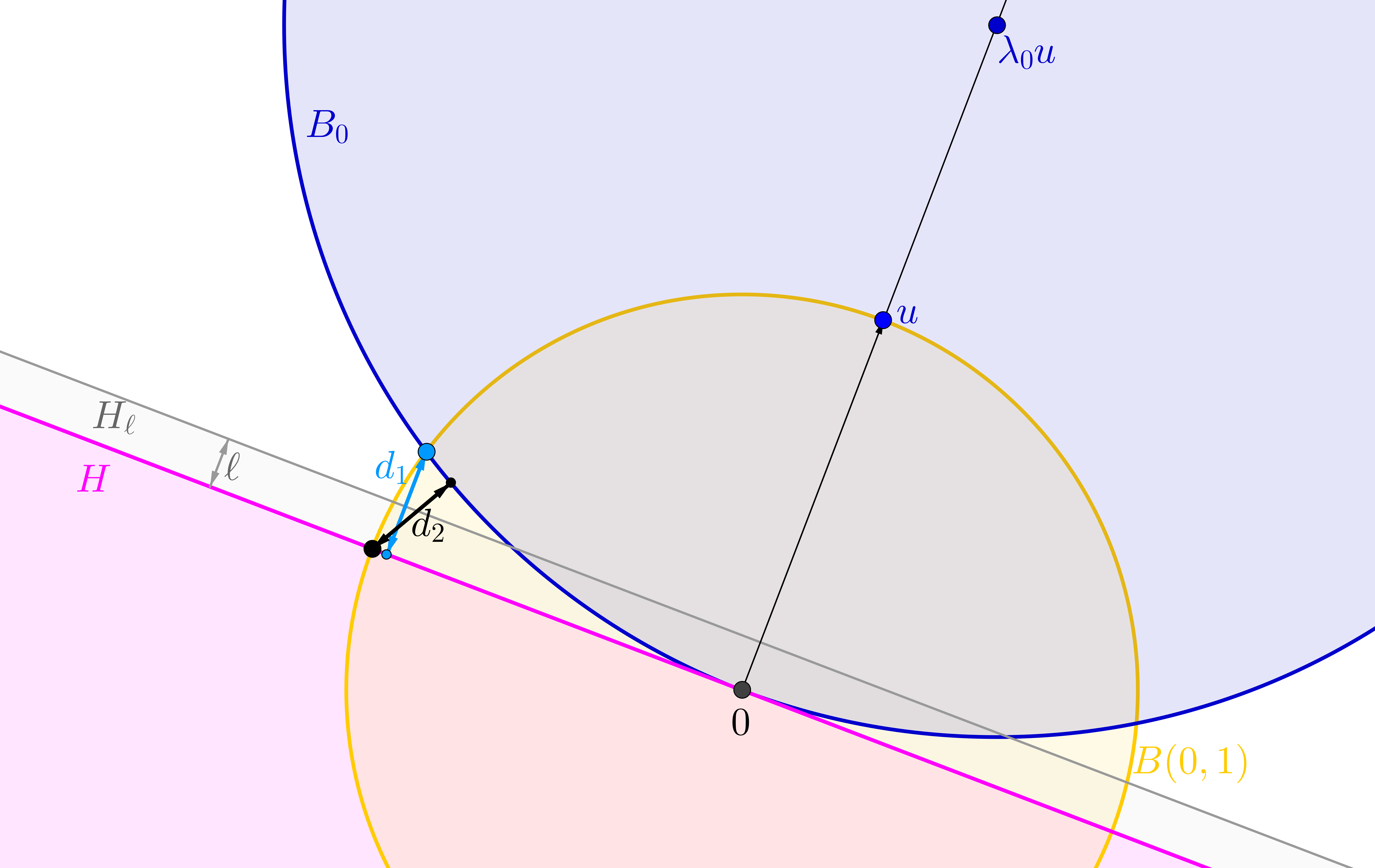}
    \caption{The open \emph{blue} ball, $B_0$, is used to define the finite set $A$. 
    The open \emph{yellow} ball, $B(0,1)$, does not contain any integer points apart from $0$.
    Assuming $\tau \leq \min ( d_1 , d_2)$, the points of $A$ and their perturbations avoid the \emph{magenta} half-space, $H$, which is necessary to apply Lemma~\ref{lem:stability_of_crossing}. 
    In fact, in order to remove the dependence of Lemma~\ref{lem:stability_of_crossing} on $u$ and $A$, we need to choose $\tau$ even smaller to avoid the \emph{gray} half-space, $H_\ell$.
    Note that the distances $d_1, d_2$ remain constant as we rotate $u$.
    }
  \label{fig:proof4dot3}
\end{figure}

  Similarly, points $p=\varphi(a)\in B_0$ satisfy $\scalprod{u}{p} > 0 $ trivially, but we need to ascertain $\scalprod{u}{a} > 0$ by showing that every $a \in H \cap \Zspace^d \setminus \{ 0 \} \subseteq H \setminus B(0,1)$ is further than $\tau$ from $B_0$. 
  Assuming $\tau$ is smaller than 
  $d_2=\inf \{ d(x,B_0)\ | \ x\in H \setminus B(0,1)\}$,
  no point in $a \in \Zspace^d \setminus \{ 0 \}$ with $\scalprod{u}{a} \leq 0$ can be perturbed into $B_0$.
  Again, note that the condition on $\tau$ depends neither on $u$ nor on $\varphi$. 
  Hence, Conditions~1 and 3 and therefore all four conditions are satisfied.
  
  \smallskip
  For the last step, let $n = \card{A}$ and write $a_1, a_2, \ldots, a_n$ for the points in $A$.
  Let $x_i = \lambda (a_i)$ and $y_i = \lambda (\varphi(a_i))$, and assume that $\tau > 0$ was chosen so that Lemma~\ref{lem:stability_of_crossing} implies $|x_i - y_i| < \ee$. 
  Here we use Lemma~\ref{lem:stability_of_crossing} twice: once for $a = a_i$ and $p = \varphi (a_i)$ and the second time for $a = \varphi (a_i)$ and $p = a_i$, with the second application justified by Condition~3.
  We notice that the condition on $\tau$ can be made independently of $u$ and $\varphi$, by observing that the bounds needed to satisfy Conditions~1 and 3 can be strengthened: 
  by setting $\ell = \frac{1}{2} \min \{ d_1 , d_2 \}$ and assuming $\tau$ is smaller than $\inf \{ d(x,H_\ell)\ | \ x\in B_0 \setminus B(0,1)\}$, in which $H_\ell$ is the half-space of points $y$ satisfying $\scalprod{u}{y} \leq \ell $, and assuming $\tau$ is smaller than $\inf \{ d(x,B_0)\ | \ x\in H_\ell \setminus B(0,1)\}$, we can
  lower bound $\scalprod{u}{a},\scalprod{u}{\varphi(a)}, \norm{a}, \norm{\varphi(a)}$ by $\ell$ for all $a \in A$. 
  In addition, $2\lambda_0+1$ is an upper bound for $\norm{a}, \norm{\varphi(a)}$ for all $a \in A$. 
  With these bounds, we can choose $\tau$ small enough, independently from $u$ and $\varphi$, so that Lemma~\ref{lem:stability_of_crossing} holds.

  Lemma~\ref{lem:stability_of_rank} now implies $|\beta_k(u) - \alpha_k (u)| < \ee$.
  Since all conditions on $\tau$ depend only on $d,k,\ee$ and not on $u,\varphi$, this completes the proof of the claim for any perturbation $\varphi$ with magnitude less than $\tau$ and for any direction $u$.
\end{proof}
In words, Theorem \ref{thm:stability_of_Brillouin_zones} asserts that the difference between the distances of the outer boundaries of the $k$-th Brillouin zones in a given direction from the origin---before and after the perturbation---goes to $0$ when $\tau$ tends to $0$.
This also holds for their inner boundaries, which are the outer boundaries of the $(k-1)$-st Brillouin zones.
Hence, the Hausdorff distance between the $k$-th Brillouin zones---before and after the perturbation---goes to $0$ when $\tau$ tends to $0$.
Theorem~\ref{thm:stability_of_Brillouin_zones} also holds for Delone sets in $\Rspace^d$.

\subsection{Stability Experimentally}
\label{sec:4.3}

We use an indirect approach to probe the stability of the Brillouin zones experimentally.
For the integer lattice, every Brillouin zone has area $1.0$, but for a perturbation, this is no longer necessarily true.
It would be interesting to know how perturbations affect the area of the zones.
Clearly, the area exchange is a zero-sum game, so we expect some oscillation around $1.0$, and this is confirmed by the graphs in the top panel of Figure~\ref{fig:area}.
As suggested by the graphs in the bottom panel of Figure~\ref{fig:area} and implied by Theorem~\ref{thm:width_for_perturbed_integer_lattices}, the average area of the first $k$ Brillouin zones of $0$ in a perturbation of $\Zspace^2$ converges to $1.0$.
\begin{figure}[htb]
  \centering
    \vspace{0.1in}
    \includegraphics[width=0.625\textwidth]{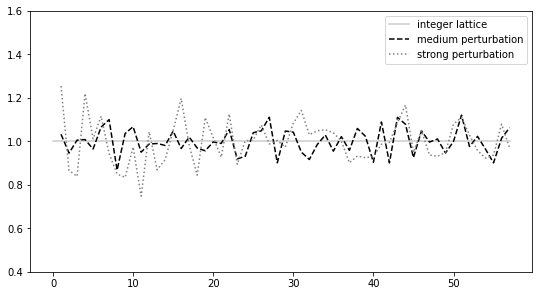} \\
    \includegraphics[width=0.625\textwidth]{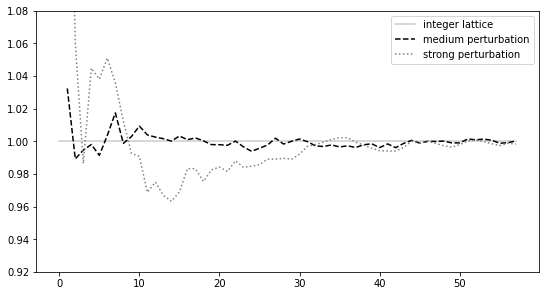} \\
    \vspace{-0.1in}
    \caption{\emph{Top:} the area of the $k$-th Brillouin zone of $0$ in $\Zspace^2$ and in two perturbations of $\Zspace^2$.
    We get progressively more noisy curves for increasing strength of the perturbation.
    \emph{Bottom:} the cumulative area of the first $k$ Brillouin zones divided by $k$.
    We get a straight line for $\Zspace^2$ and graphs that converge to it for perturbations of $\Zspace^2$. 
    Similar to Figure~\ref{fig:width}, the curves of the medium and strong perturbations may be inaccurate beyond $k = 52$ and $k = 34$, respectively.}
  \label{fig:area}
\end{figure}

We define the \emph{outer perimeter} of the $k$-th Brillouin zone as the length of the boundary of the union of the first $k$ Brillouin zones, which is $\domain{k}{0}{A}$.
By comparing it with the length of the circle bounding a disk of the same area as $\domain{k}{0}{A}$, we get the \emph{distortion} of the outer perimeter.
A recent analysis of the distortion of curves in a different context identified $4/\pi$ as a universal constant for the distortion of length \cite{EdNi21}, see also \cite{BTZ00}.
We therefore compare the distortions we get for the integer lattice and two perturbations of it with this constant in Figure~\ref{fig:perimeter}.
The findings encourage us to ask for a proof that also in our context the distortion converges to $4/\pi$.
\begin{figure}[htb]
  \centering
    \vspace{0.1in}
    \includegraphics[width=0.625\textwidth]{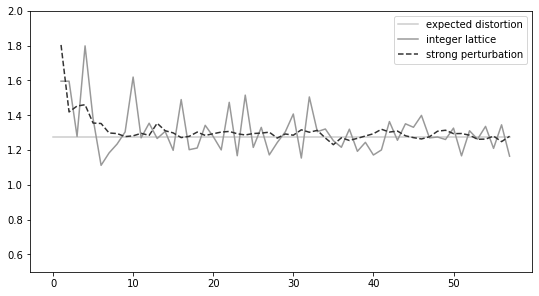}
    \vspace{-0.1in}
    \caption{The outer perimeter of the $k$-th Brillouin zone of $0 \in \Zspace^2$ divided by the perimeter of the equal area disk, the same ratio for a strong perturbation of $\Zspace^2$, and the conjectured limit of the distortion at ${4}/{\pi} = 1.27\ldots$. 
    Similar to Figure~\ref{fig:width}, the curve of the strong perturbations may be inaccurate beyond $k =34$.}
  \label{fig:perimeter}
\end{figure}

\section{Number of Chambers}
\label{sec:5}
 
 We get upper bounds on the number of chambers of individual Brillouin zones from the relation of the zones to $k$-sets and order-$k$ Voronoi tessellations.

\subsection{Integer Lattice in the Plane}
\label{sec:5.1}

As mentioned in Section~\ref{sec:2}, there is a connection between $k$-th Brillouin zones and order-$k$ Voronoi tessellations, which we exploit to get a linear upper bound on the number of chambers in the $2$-dimensional lattice case.
We also prove a lower bound, for which we need an extension of a well known number theoretic fact.
We begin with this extension.
\begin{lemma}[\#Integer Points on Circle]
  \label{lem:integer_points_on_circle}
  For every $\ee > 0$, every circle of radius $R$ passes through at most $O(R^\ee)$ integer points.
\end{lemma}
\begin{proof}
  Let $\Gamma$ be a circle with radius $R$ and assume first that it is centered at the origin in $\Rspace^2$.
  Let $n (R)$ be the number of integer points on $\Gamma$, which is the number of ways we can write $R^2$ as a sum of two squares.
  It is known that $n (R)$ is at most some constant times the number of divisors of $R$ \cite[Lemma 2]{CoHi07}, which implies $n (R) = O(R^\ee)$ for every $\ee > 0$; see \cite{HaWr08}.

  \smallskip
  Next consider the more general case, in which the center of $\Gamma$ is not an integer point.
  We can assume that $\Gamma$ passes through at least three integer points, else there is nothing to prove.
  We may also assume that one of these three points is $0 = (0,0)$, and we write $a = (a_1, a_2)$ and $b = (b_1, b_2)$ for the other two integer points.
  The center of $\Gamma$ is at the intersection of the two bisectors defined by $0, a$ and by $0, b$.
  Equivalently, its coordinates are the solutions to the linear system
  \begin{align}
    2 a_1 x_1 + 2 a_2 x_2  &=  a_1^2 + a_2^2 , 
      \label{eqn:coordinate1} \\
    2 b_1 x_1 + 2 b_2 x_2  &=  b_1^2 + b_2^2 .
      \label{eqn:coordinate2}
  \end{align}
  Using Cramer's Rule, we get the coordinates of the center as ratios of integer determinants, and the two ratios share the denominator, which is $\Delta = 2 a_1 b_2 - 2 a_2 b_1$.
  By assumption, the radius of $\Gamma$ is $R$, which implies that the coordinates are at most $2R$, and thus $|\Delta| \leq 16 R^2$.
  Scaling $\Gamma$ by $\Delta$, we get a new circle, with radius at most $16 R^3$, whose center is an integer point.
  The integer points on $\Gamma$ map to integer points on the new circle, but there can be only $n (\Delta R)$ such integer points, which is $O (R^\ee)$, for every $\ee > 0$.
\end{proof}
\begin{theorem}[\#Chambers for 2D Integer Lattice]
  \label{thm:chambers_for_2D_integer_lattice}
  For every $k \geq 2$ and $\ee>0$, the number of chambers in the $k$-th Brillouin zone of $0 \in \Zspace^2$ is at least $\Omega (k^{1-\ee})$ and at most $6k-6$.
\end{theorem}
\begin{proof}
  We first prove the lower bound.
  Let $x$ be an interior point of a chamber in the $k$-th Brillouin zone of $0$, and write $A(x)$ for the $k-1$ integer points in the interior of $B(x, \norm{x})$.
  The point $-x$ is interior to the diametrically opposite chamber, and $A(-x)$ is another set of $k-1$ integer points.
  Since $B(x,\norm{x})$ and $B(-x,\norm{x})$ have disjoint interiors, $A(x)$ and $A(-x)$ are disjoint.
  
  For $k \geq 2$, the $k$-th Brillouin zone consists of a cyclic ring of chambers, each sharing a vertex with its predecessor and another vertex with its successor along the cyclic order.
  For each chamber, we get a set of $k-1$ integer points, and we consider the sets of two consecutive chambers that share a vertex $y$.
  If a point $a \in \Zspace^2$ belongs to one set but not the other, then the bisector of $0$ and $a$ passes through $y$.
  The integer points whose bisectors with $0$ pass through $y$ all lie on the circle centered at $y$ and passing through $0$.
  By Theorem~\ref{thm:width_for_integer_lattices}, the radius of this circle is at most some constant times $\sqrt{k}$, so Lemma \ref{lem:integer_points_on_circle} implies that there are at most $O (k^\ee)$ integer points on this circle, for any $\ee > 0$.
  The claimed lower bound follows because we need at least $\Omega (k^{1-\ee})$ steps between adjacent chambers of the $k$-th Brillouin zone to exchange all points, which is necessary to travel from $A(x)$ to $A(-x)$.
  
  \smallskip
  We second prove the upper bound.
  Let $\rho$ be the radius, which we will specify later, and consider the integer lattice within the ball $B(0, \rho)$.
  Set $A = B(0, \rho) \cap \Zspace^2$ and $n = \card{A}$, and note that \eqref{eqn:points} implies $\pi [\rho - \sfrac{\sqrt{2}}{2}]^2 < n < \pi [\rho + \sfrac{\sqrt{2}}{2}]^2$.
  Write $\Brillouin{k}{A}$ for the order-$k$ Brillouin tessellation of $A$ defined in Section~\ref{sec:2}.
  For $k = 1$, this is the ordinary Voronoi tessellation with one ($2$-dimensional) region per point, and for $k \geq 2$, Corollary~\ref{cor:order-k_Brillouin_tessellation} asserts that the number of regions is less than $[6k-6]n$.
  
  For $a \in \Zspace^2$, we call a closed ball \emph{$a$-anchored} if $a$ belongs to its boundary, and we call a point, $b \in \Zspace^2$, \emph{$k$-near} to $a$ if there is an $a$-anchored ball, $B$, with $b \in \interior{B}$ and $\card{(B \cap \Zspace^2)} \leq k-1$.
  Now recall that $\Brillouin{k}{A}$ is obtained by drawing the $k$-th Brillouin zones of all points in $A$ next to each other.
  The $k$-th Brillouin zone of $a$ in $A$ is the same as in $\Zspace^2$ if all $k$-near points of $a$ in $\Zspace^2$ also belong to $A$.
  By \eqref{eqn:maxdist-integer-lattice}, the maximum distance of $a \in \Zspace^2$ from a point in its $k$-th Brillouin zone is $R_k < \sqrt{\sfrac{k}{\pi}} + \sfrac{\sqrt{2}}{2}$.
  It follows that the disk $B(a, 2 R_k)$ contains all $k$-near points of $a$.
  Hence, for all points $a \in B(0, \rho - 2R_k) \cap \Zspace^2$, the $k$-th Brillouin zone of $a$ in $A$ is the same as in $\Zspace^2$.
  Writing $n_0$ for the number of such points, we use again a volume argument to see that $n_0 > \pi [\rho - 2R_k - \sfrac{\sqrt{2}}{2}]^2$.
  The number of remaining points in $A$ is of lower order.
  To do the final counting, let $\chi_k$ be the number of chambers in the $k$-th Brillouin zone of a point in $\Zspace^2$.
  The total number of chambers we get for the $n_0$ points is less than the number of regions in $\Brillouin{k}{A}$.
  Assuming $k \geq 2$, this implies
  \begin{align}
    \chi_k  &<  [6k-6] \frac{n}{n_0}
             <  [6k-6] \frac{\pi [\rho + \sfrac{\sqrt{2}}{2}]^2}
                            {\pi [\rho - 2R_k - \sfrac{\sqrt{2}}{2}]^2} .
  \end{align}
  We can make the ratio as close to $1$ as we like by choosing $\rho$ as large as we like.
  This finally implies $\chi_k \leq 6k-6$.
\end{proof}

The upper bound on the number of chambers extends without adjustment of constants to general lattices in $\Rspace^2$.
For a lattice in $\Rspace^2$ that contains no four points on a common circle, the lower bound argument in the above proof implies a matching linear lower bound.
Contrary to lattices, a locally finite point set can have a Brillouin zone with infinitely many chambers; see \cite[Abbildung 3.4]{Voi08} for an example.

\subsection{Integer Lattices Beyond Two Dimensions}
\label{sec:5.2}

The proof of the upper bound in Theorem~\ref{thm:chambers_for_2D_integer_lattice} uses Corollary~\ref{cor:order-k_Brillouin_tessellation}, which only holds in $\Rspace^2$.
To generalize to higher dimension we need a different strategy.
Using inversion, we transform the problem of counting chambers of the $k$-th Brillouin zone to counting $k$-sets of a finite set.
We begin by proving that sets of $O(k)$ points suffice for the transformation.
Recall that a point $a \in \Zspace^d$ is \emph{$k$-near} to $0$ if there exists a $0$-anchored ball, $B$, such that $a \in \interior{B}$ and $\card{(\interior{B} \cap \Zspace^d)} \leq k-1$.
Theorem \ref{thm:width_for_integer_lattices} implies an upper bound on the number of $k$-near points.
\begin{figure}[hbt]
  \centering
    \vspace{0.0in}
    \resizebox{!}{2.8in}{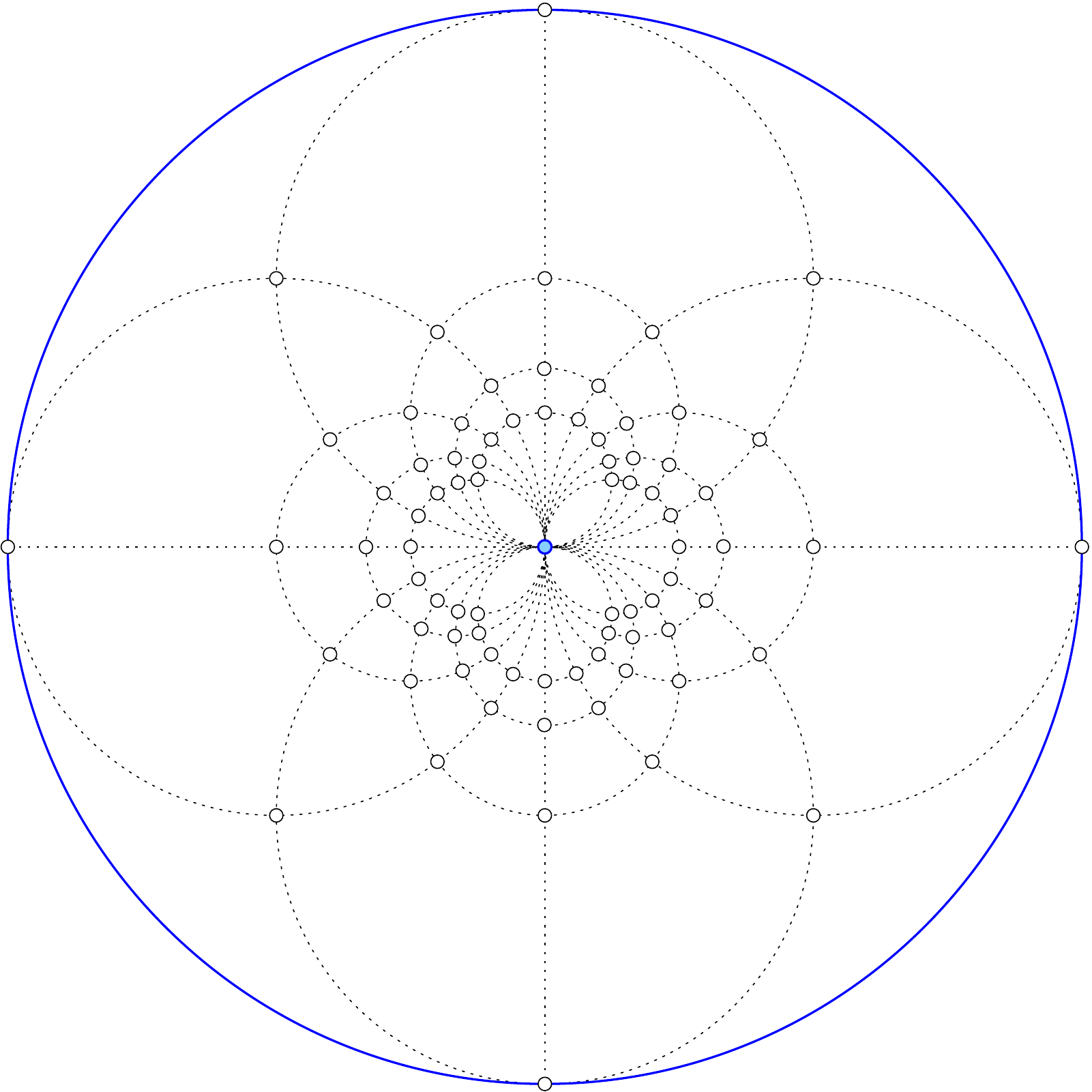}
    \caption{The image of the integer lattice under inversion.
    The dotted lines and circles are the images of nine vertical and nine horizontal integer lines.
    The \emph{blue} unit circle is preserved by the inversion.
    Since all integer points other than $0$ lie on or outside the unit circle, their images under the inversion all lie on or inside the unit circle.}
  \label{fig:inverse}
\end{figure}
\begin{lemma}[\#Near Points]
  \label{lem:near_points}
  For every integer $k \geq 1$, the number of $k$-near points of $0 \in \Zspace^d$ is $O(k)$, assuming $d$ is constant.
\end{lemma}
\begin{proof}
  The center of any $0$-anchored closed ball with fewer than $k$ integer points in the interior belongs to one of the first $k$ Brillouin zones of $0$.
  By \eqref{eqn:maxdist-integer-lattice}, the distance of such a point from $0$ is at most $R_k < \sqrt[d]{\sfrac{k}{\nu_d}} + \sfrac{\sqrt{d}}{2}$.
  It follows that all $k$-near points are contained in $B(0, 2R_k)$.
  By a straightforward volume argument, the number of such points is bounded from above by the volume of $B(0, 2R_k + \sfrac{\sqrt{d}}{2})$, which is
  \begin{align}
    \nu_d \left[ 2 R_k + \sfrac{\sqrt{d}}{2} \right]^d
    &<  \nu_d \left[ 2 \sqrt[d]{\sfrac{k}{\nu_d}} + \tfrac{3}{2} \sqrt{d} \right]^d .
  \end{align}
  Since $d$ is a constant, this volume is $O(k)$, so the number of $k$-near points is $O(k)$.
\end{proof}

For the next step, we \emph{invert} $\Rspace^d$ in the unit sphere using the map $\iota \colon \Rspace^d \setminus \{0\} \to \Rspace^d$ defined by $\iota (x) = x / \norm{x}^2$; see Figure~\ref{fig:inverse}.
It preserves the unit sphere and exchanges points inside with points outside this sphere.
Importantly, it maps every $0$-anchored ball to a closed half-space that does not contain $0$.
In particular, it maps a $0$-anchored ball, $B$, with $\card{(\interior{B} \cap \Zspace^d)} = k-1$ to a half-space, $\iota (B)$, that contains $k-1$ points of $A = \iota ( \Zspace^d \setminus \{0\} )$ in its interior.
We call these points a \emph{$(k-1)$-set} of $A$.
By Lemma~\ref{lem:near_points}, there is a subset $A' \subseteq A$ of size $\card{A'} = O(k)$ such that every $(k-1)$-set of $A$ is also a $(k-1)$-set of $A'$.
This is interesting because counting $k$-sets is a much studied while poorly understood problem in discrete geometry.
Nevertheless, non-trivial bounds on the maximum number of $k$-sets are known in all finite dimensions, and the inversion together with Lemma~\ref{lem:near_points} implies similar bounds for the number of chambers of Brillouin zones.
Let $\ksets{k}{d}{n}$ be the maximum number of $k$-sets any set of $n$ points in $\Rspace^d$ can have.
\begin{theorem}[\#Chambers for Integer Lattices]
  \label{thm:chambers_for_integer_lattices}
  For every $k \geq 1$, there exists a constant, $C$, depending on $d$, such that $\ksets{k-1}{d}{C k}$ is an upper bound on the number of chambers of the $k$-th Brillouin zone of $0 \in \Zspace^d$.
\end{theorem}
At the time of writing this paper, the best known upper bounds for $\ksets{k}{d}{n}$ in which $n = O(k)$ are $\ksets{k}{2}{n} = O(n^{4/3})$ in \cite{Dey98}, $\ksets{k}{3}{n} = O(n^{5/2})$ in \cite{SST01}, and $\ksets{k}{d}{n} = O(n^{d-c_d})$ with $c_d > 0$ small and tending to $0$ as $d$ grows in \cite{ABFK92}.
The bounds in Theorem~\ref{thm:chambers_for_integer_lattices} extend to Delone sets, so it is humbling that we did not get better bounds even for the integer lattices beyond two dimensions.

\subsection{Number of Chambers Experimentally}
\label{sec:5.3}

We illustrate Theorem~\ref{thm:chambers_for_2D_integer_lattice} by showing the number of chambers of the $k$-th Brillouin zone of $0 \in \Zspace^2$ in the top panel of Figure~\ref{fig:chambers}.
As predicted, the curve stays below the straight line of the upper bound for $k\geq 2$.
By comparison, the number of chambers for a perturbation of $\Zspace^2$ goes sometimes above this bound, which is not a contradiction since perturbations are not necessarily lattices so the bound does not apply.
The bottom panel of Figure~\ref{fig:chambers} shows the cumulative number of chambers for the first $k$ Brillouin zones.
The linear upper bound in the top panel turns into a quadratic upper bound in the bottom panel.
The curve for $\Zspace^2$ stays clearly below that bound.
\begin{figure}[hbt]
  \centering
    \vspace{0.0in}
    \includegraphics[width=0.625\textwidth]{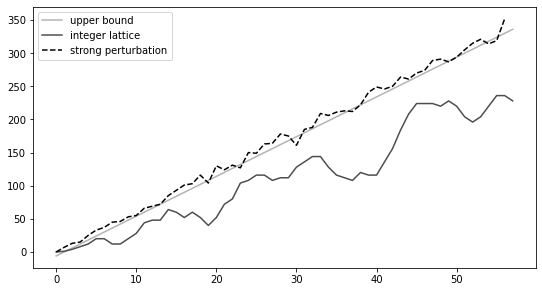} \\
    \hspace{-0.15in}
    \includegraphics[width=0.641\textwidth]{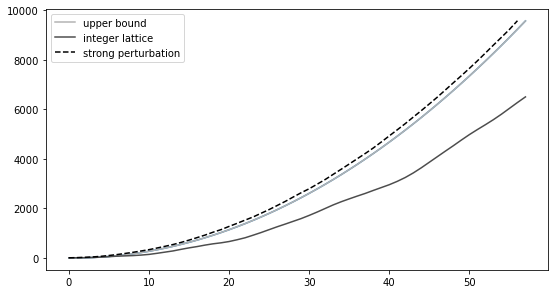}
    \caption{\emph{Top:} the number of chambers in the $k$-th  Brillouin zone of $0 \in \Zspace^2$ together with the linear upper bound and the numbers for a strong perturbation, for comparison.
    \emph{Bottom:} the cumulative number of chambers in the first $k$ Brillouin zones of $0 \in \Zspace^2$ together with the quadratic upper bound and the cumulative numbers for the strong perturbation.
    Similar to Figure~\ref{fig:width}, the curve of the strong perturbations may be inaccurate beyond $k =34$.}
  \label{fig:chambers}
\end{figure}

\section{Size of a Chamber}
\label{sec:6}

In this section, we prove bounds on the maximum diameter of a chamber in a $k$-th Brillouin zone.
Bounds on the volume and other measures follow.

\subsection{Integer Lattices}
\label{sec:6.1}

The integer lattice forces the chambers of the $k$-th Brillouin zone to shrink when $k$ increases.
To prove a quantified version of this claim, we begin with an exercise in elementary geometry.
While the application of the result from this exercise is in $d$ dimensions, it is convenient to temporarily add another dimension and embed $\Rspace^d$ in $\Rspace^{d+1}$.

\smallskip
Let $c$ be a positive constant, $R \geq c$ a radius, and $\Sigma_0$ a $d$-sphere with radius $\rho_0 = R$ in $\Rspace^{d+1}$.
We construct a $(d-1)$-sphere $\Sigma_1$ by slicing $\Sigma_0$ with a $d$-plane at distance $\rho_0 - c$ from the center of $\Sigma_0$.
It is easy to check that the radius of $\Sigma_1$ is $\rho_1 = \sqrt{2c \rho_0 - c^2}$, which is necessarily at least $c$.
We iterate and thus get a sequence of spheres $\Sigma_i$, for $0 \leq i \leq d$, in which the dimension of $\Sigma_i$ is $d-i$, and the radius of every subsequent sphere is $\rho_i = \sqrt{2c \rho_{i-1} - c^2} \geq c$.
Writing the radii in terms of $R$, we get
\begin{align}
  \rho_1  &=  \sqrt{2cR-c^2}  ~~~~~~~~~~~~<  (2c)^{1/2} R^{1/2} , 
    \label{eqn:r2} \\
  \rho_2  &=  \sqrt{2c \sqrt{2cR-c^2} - c^2} <  (2c)^{3/4} R^{1/4} ,
    \label{eqn:r3}
\end{align}
and more generally $\rho_i \leq (2c)^{1-1/2^{i}} R^{1/2^{i}}$ for $0 \leq i \leq d$.
Write $z_i$ for the center of $\Sigma_i$ and observe that its distance to the closest point on $\Sigma_0$ is
\begin{align}
  \delta_i  &= R - \Edist{z_i}{z_0} = R - \sqrt{R^2 - \rho_{i}^2}
        =  R - \sqrt{\left(R - \frac{\rho_{i}^2}{2R}\right)^2 - \frac{\rho_{i}^4}{4 R^2}} \\
       &\leq  R - \left( R - \frac{\rho_i^2}{2R} \right) + \frac{\rho_i^2}{2R}
        \leq  \frac{\rho_i^2}{R}
       \leq  \frac{4c^2}{R} \left( \frac{R}{2c} \right)^{1/2^{i-1}}
        =  2c \left( \frac{2c}{R} \right)^{1-1/2^{i-1}} .
\end{align}
Observe that all $d+1$ centers lie in a $d$-plane that intersects $\Sigma_0$ in a $(d-1)$-sphere of radius $R$, which we denote $\Sigma_0'$.
It follows that $\delta_i$ is also the distance of $z_i$ to the closest point of $\Sigma_0'$, for $0 \leq i \leq d$.
Of particular interest is the last center, $z_d$, and its distance to $\Sigma_0'$, which is $\delta_d$.
The $d$-plane that contains $\Sigma_0'$ is where we apply the insights from the exercise.

\smallskip 
We claim that in $\Rspace^d$, there are necessarily many integer points at distance at most $\delta_d$ from $\Sigma_0$.
To formulate this claim, we call a point on a $(d-1)$-sphere in $\Rspace^d$ a \emph{pole} if there is an index $1 \leq i \leq d$ such that the point has either minimum or maximum $i$-th coordinate among all points of the sphere.
A $(d-1)$-sphere has $2d$ poles.
For the next definition, we let $\Rspace^d$ be the $d$-plane spanned by the first $d$ coordinate vectors of $\Rspace^{d+1}$.
An \emph{integer $(i+1)$-plane} is an axis-parallel $(i+1)$-dimensional plane in $\Rspace^{d+1}$ normal to $\Rspace^d$ passing through at least one integer point.
It is normal to $d-i$ of the first $d$ coordinate axes and determined by the corresponding $d-i$ values, which are integers.

The largest open ball in $\Rspace^d$ that contains no point in $\Zspace^d$ has radius $\sfrac{\sqrt{d}}{2}$.
This implies that for a $(d-1)$-sphere with radius $R \geq \sqrt{d}$ in $\Rspace^d$, and for any coordinate axis, there exists an integer $d$-plane orthogonal to that axis at distance at least $R - \sqrt{d}$ from the center that contains at least one point of $\Zspace^d$ on or inside the sphere.
This motivates us to set $c = \sqrt{d}$.
\begin{lemma}[Nearby Integer Points]
  \label{lem:nearby_integer_points}
  For every $(d-1)$-sphere with radius $R \geq \sqrt{d}$ in $\Rspace^d$, and every one of its poles, there is an integer point at distance at most $R$ to the center, at most $\sqrt{2\sqrt{d}R}$ to the pole, and at most 
  $2\sqrt{d} (2\sqrt{d}/R)^{1 - 1/2^{d-1}}$ to the closest point on the $(d-1)$-sphere.
\end{lemma}
\begin{proof}
  Let $S_0$ be a $d$-sphere with center in $\Rspace^d$ and radius $R \geq \sqrt{d}$ in $\Rspace^{d+1}$, and write $B_0$ for the closed $(d+1)$-ball whose boundary is $S_0$.
  Let $S_0' = S_0 \cap \Rspace^d$, which is a $(d-1)$-sphere of radius $R$ and thus encloses or passes through at least one point in $\Zspace^d \subseteq \Rspace^d$.
  Because of symmetry, it is sufficient to construct a point $p \in \Zspace^d$ that satisfies the conditions of the lemma for one given pole of $S_0'$.
  Let this pole be the point with maximum first coordinate among all points of $S_0'$.
  Let $h_1$ be the integer $d$-plane orthogonal to the first coordinate axis, with maximum first coordinate such that $B_0 \cap h_1 \cap \Zspace^d \neq \emptyset$.
  Let $S_1 = S_0 \cap h_1$ be the corresponding $(d-1)$-sphere, and note that its radius is $r_1 \leq \rho_1$, in which $\rho_1$ is defined with $c = \sqrt{d}$ in \eqref{eqn:r2}.
  Write $B_1 = B_0 \cap h_1$ and iterate to get a sequence of $d+1$ spheres, as before.
  By construction, the radii of these spheres satisfy $r_i \leq \rho_i$, and their centers have distance $d_i \leq \delta_i$ to their closest points on $S_0'$.
  Also by construction, each $B_i$ contains at least one point in $\Zspace^d$, and since its center maximizes the distance to $S_0'$, the distance of this integer point to the closest point on $S_0'$ is at most $d_i \leq \delta_i$.
  The case $i=d$ proves the claimed bounds on the distance to $S_0'$ and to the center of $S_0'$.
  
  \smallskip
  To also bound the distance to the pole, observe that $p \in B_1$, which is a $d$-ball with radius $r_1 \leq \rho_1 = \sqrt{2\sqrt{d}R-d}$.
  The distance of the pole to $z_1$, the center of the $d$-ball, is at most $\sqrt{d}$, so the distance to $p$ is at most $\sqrt{r_1^2 + d} \leq \sqrt{2\sqrt{d}R}$.
\end{proof}

Let now $S(x,r)$ and $S(y,R)$ be two $(d-1)$-spheres with centers $x, y \in \Rspace^d$ and radii $r \leq R$.
Assuming $x \neq y$, an entire hemisphere of $S(y,R)$ lies outside the other sphere, namely the hemisphere of points $a \in S(y,R)$ that satisfy $\scalprod{a-y}{y-x} \geq 0$.
We prove that this hemisphere contains a large cap within which all points have distance at least some constant times $\Edist{x}{y}$ to the closest point of $S(x,r)$.
Specifically, if $a \in S(y,R)$ is a point in this hemisphere and $\theta \leq {\pi}/{2}$ is the angle between $a$ and the central point, $z$, of the hemisphere, as in Figure~\ref{fig:cone}, then
\begin{align}
  \Edist{a}{x}^2  &=  R^2 \sin^2 \theta + (R \cos \theta + \Edist{x}{y} )^2 \\
  &=  R^2 + 2R \Edist{x}{y} \cos \theta + \Edist{x}{y}^2
  \geq  (R + \Edist{x}{y} \cos \theta)^2 .
\end{align}
The distance of $a$ to the closest point of $S(x,r)$ is therefore
$\Edist{a}{x} - r  \geq  \Edist{a}{x} - R  \geq  \Edist{x}{y} \cos \theta$.
In words, for every constant angle $\theta < \frac{\pi}{2}$, the distance between $a$ and the closest point of $S(x,r)$ is at least some constant fraction of the distance between the centers.
\begin{figure}[hbt]
  \centering
  \vspace{0.1in}
  \resizebox{!}{2.2in}{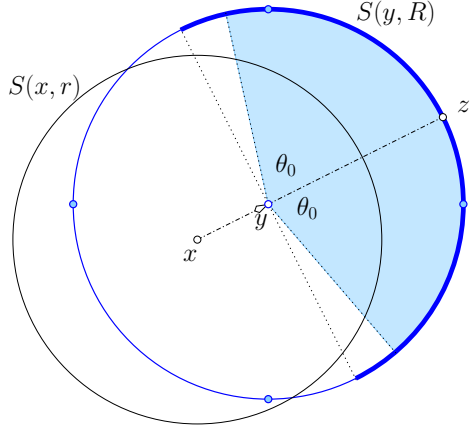}
  \caption{By assumption, $S(y,R)$ is no smaller than $S(x,r)$.
  The half-circle of points $a \in S(y,R)$ with $\scalprod{a-y}{y-x} \geq 0$ is highlighted.
  The cone with angle $\theta_0 = \arccos \sfrac{\sqrt{2}}{6}$ is guaranteed to contain at least one of the four poles of $S(y,R)$ together with the nearby integer point.}
  \label{fig:cone}
\end{figure}

In the following, we will consider cones, for which we introduce the following nomenclature.
By a \emph{cone} with \emph{apex} $y \in \Rspace^d$, \emph{direction} $z-y \in \Rspace^d \setminus \{0\}$, and \emph{angle} $\theta_0 < \frac{\pi}{2}$, we mean the set of points $v \in \Rspace^d$ such that the angle between $v-y$ and $z-y$ is at most $\theta_0$; see Figure~\ref{fig:cone}.
We will need $\theta_0$ such that at least one of the integer points that satisfy the conditions of Lemma~\ref{lem:nearby_integer_points} is contained in this cone.
For every direction $z-y$, there is a pole $b$ of $S(y,R)$ with $\scalprod{z-y}{b-y} \geq R^2 / \sqrt{d}$,
and this is best possible because $\scalprod{z-y}{b-y} \leq R^2 / \sqrt{d}$ for all poles if $z-y$ is the diagonal direction of any of the $2^d$ orthants.
Thus any cone with angle at least $\arccos \sfrac{1}{\sqrt{d}}$ contains at least one pole.
We choose the threshold a bit larger so that any such cone also contains the nearby integer point.
\begin{lemma}[Poles and Cones]
  \label{lem:poles_and_cones}
  Let $\Sigma$ be a sphere with radius $R \geq (18 d - \frac{1}{2}) \sqrt{d}$.
  Then every cone with apex at the center of $\Sigma$ and angle $\theta_0 = \arccos \sfrac{1}{(3 \sqrt{d})}$ contains at least one pole of $\Sigma$ together with the nearby integer point as constructed in the proof of Lemma~\ref{lem:nearby_integer_points}.
\end{lemma}
\begin{proof}
  Write $y$ for the center of $\Sigma$ and let $z$ be an arbitrary point on the sphere. Before showing the claim for the cone with apex $y$, direction $z-y$ and angle $\theta_0$ (below denoted by $\theta_0$-cone), we consider a cone with the same apex and direction but a smaller angle $\varphi$ (below denoted by $\varphi$-cone):
  As mentioned above, every cone with angle $\varphi = \arccos \sfrac{1}{\sqrt{d}}$ contains at least one pole, $b$.
  Let $\Sigma_2$ be the $(d-2)$-sphere obtained by slicing $\Sigma$ with the $(d-1)$-plane of points $a$ that satisfy $\scalprod{a-y}{b-y} = R (R-\sqrt{d})$.
  The radius of $\Sigma_2$ is $\rho_2 = \sqrt{2 \sqrt{d} R - d}$.
  We need to show that $\theta_0 - \varphi$ is large enough for the $\theta_0$-cone to not only contain $b$, but also $\Sigma_2$, and thus also the integer point nearby.
  Using the triangle inequality on the unit sphere, we observe that it suffices to show that the angle of the cone defined by $\Sigma_2$ is less than $\theta_0 - \varphi$.
  We begin by computing the squared (Hausdorff) distance between the $(d-2)$-spheres at which the boundaries of the cones with angles $\varphi$ and $\theta_0$ (in the same direction) intersect $\Sigma$.
  By intersecting with a plane that passes through $z-y$, we see that this is the same as the squared distance 
  between the points $( R \cos \varphi, R \sin \varphi)$ and $(R \cos \theta_0, R \sin \theta_0)$, which is
  \begin{align}
    \!\!R^2 (\cos \varphi \!-\! \cos \theta_0)^2 \!+\! R^2 (\sin \theta_0 \!-\! \sin \varphi)^2
    &\geq  R^2 (\cos \varphi \!-\! \cos \theta_0)^2
     \geq  \left( \tfrac{R}{\sqrt{d}} \!-\! \tfrac{R}{3 \sqrt{d}} \right)^2  =  \tfrac{4 R^2}{9d} .
  \end{align}
  With the assumed lower bound on $R$, this is at least $4$ times the squared radius of $\Sigma_2$, which is $\rho_2^2 = 2 \sqrt{d} R - d$.
  Hence $\rho_2$ is at most half the (Hausdorff) distance between the two $(d-2)$-spheres, which implies that the cone with angle $\theta_0$ contains $\Sigma_2$.
\end{proof}
We are now ready to prove the main result of this subsection.
\begin{theorem}[Size for Integer Lattices]
  \label{thm:size_for_integer_lattices}
  For every $d \geq 2$, every chamber in the $k$-th Brillouin zone of $0 \in \Zspace^d$ has diameter at most $O( k^{(-1 + 1/ 2^{d-1})/d} )$.
  Specifically, for $k \geq \nu_d (18 d \sqrt{d})^d$, the diameter is bounded from above by $(18 d \sqrt{d}) \cdot {\nu_d}^{1/d} \cdot k^{(-1 + 1/2^{d-1})/d}$.
\end{theorem}
\begin{proof}
  Let $x$ and $y$ be points in the $k$-th Brillouin zone of $0$, and let $S(x,r)$ and $S(y,R)$ be the spheres with centers $x, y \in \Rspace^d$ and radii $r = \norm{x}$ and $R = \norm{y}$.
  Assuming the distance between $x$ and $y$ exceeds the claimed upper bound on the diameter, we prove that $x$ and $y$ belong to different chambers by showing that $S(x,r)$ and $S(y,R)$ enclose different sets of integer points.
  We get $r, R > \sqrt[d]{\sfrac{k}{\nu_d}} - \sfrac{\sqrt{d}}{2}$ from Theorem~\ref{thm:width_for_integer_lattices}.
  By the assumed lower bound on $k$, we have $\sqrt[d]{\sfrac{k}{\nu_d}} \geq 18 d \sqrt{d}$ and therefore $R \geq 18 d \sqrt{d} - \frac{1}{2} \sqrt{d}$, so we can apply Lemma~\ref{lem:poles_and_cones}.
  Furthermore, $R > \frac{2}{3} \sqrt[d]{\sfrac{k}{\nu_d}}$, although the constant could be improved.
  
  \smallskip
  Assuming $r \leq R$, the hemisphere of points $a \in S(y,R)$ with $\scalprod{a-y}{y-x} \geq 0$ lies outside $S(x,r)$.
  By Lemma~\ref{lem:nearby_integer_points}, there are integer points enclosed by $S(y,R)$ whose distance from the sphere is at most
  \begin{align}
    2 \sqrt{d} \left( \frac{2 \sqrt{d}}{R} \right)^{1-\frac{1}{2^{d-1}}}
    &<  2 \sqrt{d} \left( \frac{3 \sqrt{d} \sqrt[d]{\nu_d}}{\sqrt[d]{k}} \right)^{1-\frac{1}{2^{d-1}}}
    <  \frac{D_k}{3 \sqrt{d}} ,
    \label{eqn:distancebound}
  \end{align}
  in which $D_k = (18 d \sqrt{d}) \cdot {\nu_d}^{1/d} \cdot k^{(-1 + 1/ 2^{d-1})/d}$.
  Note that $\Edist{x}{y} > D_k$, by assumption.
  The extra factor is $\sfrac{1}{(3 \sqrt{d})} = \cos \theta_0$, so there exists an integer point on or enclosed by $S(y,R)$ that lies outside $S(x,r)$.
  Hence $x$ and $y$ lie in different chambers, which implies the claimed upper bound on the diameter.
\end{proof}

A chamber with diameter $D$ can be enclosed in a cube with edges of length $D$.
This implies that the volume of the chamber is bounded by $D^d$.
Theorem~\ref{thm:size_for_integer_lattices} thus implies that the volume of a chamber in the $k$-th Brillouin zone of $0 \in \Zspace^d$ is at most $O(k^{-1 + 1/2^{d-1}})$.

\smallskip
We can also get lower bounds for the maximum diameter and volume.
For example, $\BZone{k}{0}{\Zspace^2}$ has unit area and by Theorem~\ref{thm:chambers_for_2D_integer_lattice} consists of $O(k)$ chambers.
It follows that the average and therefore also the maximum area of a chamber in this zone is $\Omega (k^{-1})$.
By Lemma~\ref{lem:thickened_sphere} and Theorem~\ref{thm:width_for_integer_lattices}, the $k$-th Brillouin zone surrounds $0$ at a distance about $\sqrt{\sfrac{k}{\pi}}$ from $0$.
The sum of diameters is therefore at least $\Omega (k^{\sfrac{1}{2}})$, which implies that the average and therefore the maximum diameter of a chamber is $\Omega (k^{\sfrac{-1}{2}})$.
Using Theorem~\ref{thm:chambers_for_integer_lattices}, similar but weaker lower bounds can be obtained for integer lattices in $d \geq 3$ dimensions.

\smallskip
The authors of this paper believe that the bound in Theorem~\ref{thm:size_for_integer_lattices} extends to lattices, but it does not extend to Delone sets.
Indeed we will see shortly that the bounds break down even for perturbations of the integer lattice.

\subsection{Perturbed Integer Lattices}
\label{sec:6.2}

As proved above, 
the size of the largest chamber in the $k$-th Brillouin tessellation of the integer lattice approaches zero as $k$ goes to infinity.
This property is not necessarily shared by arbitrarily small perturbations of $\Zspace^d$. 

\begin{theorem}[Size for Perturbed Integer Lattices]
  \label{thm:size_for_perturbed_integer_lattices}
  For every $\tau < \frac{1}{2}$, there exists a perturbation $\varphi \colon \Zspace^d \to \Rspace^d$ with magnitude at most $\tau$ such that for every $k\geq 1$, there exists a point $b_k \in \varphi(\Zspace^d)$ such that its $k$-th Brillouin zone contains a chamber with diameter at least $\tau$ and volume at least $\nu_d [\frac{\tau}{2}]^d$.
\end{theorem}
\begin{proof}
  Let $\tau < \frac{1}{2}$. 
  We first define a local perturbation, $\varphi_k$, to create a large chamber in the $k$-th Brillouin zone of a specific point, $\varphi_k (0) = 0$, and for a specific integer, $k$.
  For $k=1$ no perturbation is needed because the $1$-st Brilloun zone (the Voronoi domain) of the integer lattice contains a ball of radius $\frac{1}{2}$.

  \smallskip
  For $k \geq 2$, we construct $P = \varphi_k (\Zspace^d)$ with $\varphi_k (0) = 0$ and 
  $\Edist{a}{\varphi_k (a)} \leq \tau$ for every $a \in \Zspace^d \setminus \{ 0 \}$.
  Let $x$ be an interior point of $\BZone{k}{0}{\Zspace^d}$, which implies that $k-1 \geq 1$ integer points lie in the interior of $B (x, \norm{x})$, $0$ lies on its boundary, and all other integer points lie outside the closed ball.
  We have $\norm{x} > \frac{1}{2}$, else the ball could not contain an integer point in its interior.
  We construct $P$ by moving the points $a \in \Zspace^d$ 
  other than the origin 
  orthogonally away from the sphere bounding $B (x, \norm{x})$ if their distance to the closest point on the sphere is less than $\tau$.
  Letting $\delta = | \Edist{x}{a} - \norm{x} |$ be this distance, we define
  \begin{align}
    \varphi_k (a)  &=  \left\{ \begin{array}{cl}
        a - (\tau - \delta) \cdot \tfrac{x-a}{\Edist{x}{a}}  &  \mbox{\rm if~~~~~~\,} \norm{x} < \Edist{x}{a} < \norm{x} + \tau, \\
        a + (\tau - \delta) \cdot \tfrac{x-a}{\Edist{x}{a}}  &  \mbox{\rm if~} \norm{x} - \tau < \Edist{x}{a} < \norm{x} , \\
        a     &  \mbox{\rm otherwise.}
      \end{array} \right.
  \end{align}
  The integer points with distance at least $\tau$ from the sphere remain where they are, and the others are moved to a location at distance $\tau$ from the sphere.
  Since $\tau < \frac{1}{2}$, the images of the integer points are distinct.
  Let $y \in \Rspace^d$ have distance less than $\frac{\tau}{2}$ from $x$.
  By construction, every point of $P$ in the interior of $B(x, \norm{x})$ is in the closure of $B(x, \norm{x}-\tau)$ which is included in the interior of $B(y, \norm{y})$ because the distance between the centers of these balls is smaller than the difference of the radii. 
  Similarly, every point of $P$ outside the closure of $B(x, \norm{x})$ is outside the interior of $B(x, \norm{x}+\tau)$ and thus also outside the closure of $B(y, \norm{y})$.
  Hence, $x$ and $y$ belong to the same chamber in the $k$-th Brillouin zone of $0 \in P$.
  By construction, this chamber contains a ball with radius $\frac{\tau}{2}$ and therefore has the claimed diameter and volume.

  \smallskip
  Note that the above construction can be performed for each $k$ with a different point $b_k \in \Zspace^d$ instead of $0$. 
  By Theorem~\ref{thm:width_for_perturbed_integer_lattices}, additional perturbations of the integer points outside the ball $B(b_k, \sqrt[d]{\sfrac{k}{\nu_d}} + \sfrac{\sqrt{d}}{2} + \tau )$
  cannot destroy the large chamber created by the perturbation.
  Hence, we can choose the points $b_k$ with increasing distance from each other so that the respective balls do not intersect and thus the modifications do not interfere, yielding one perturbation $\varphi$ that creates large chambers for every $k$. 
\end{proof}

\subsection{Size Experimentally}
\label{sec:6.3}

We illustrate Theorems~\ref{thm:size_for_integer_lattices} and \ref{thm:size_for_perturbed_integer_lattices} by showing the maximum area and maximum diameter of the chambers in the $k$-th Brillouin zone of $0$ in Figure~\ref{fig:maxareadiameter}.
For $\Zspace^2$, both quantities tend to zero as $k$ goes to infinity, which is consistent with the graphs in the two panels, where we multiply with $k^{1/2}$ and $k^{1/4}$, respectively.
For perturbations of $\Zspace^2$, there are examples for which both measures stay above a positive constant even for arbitrarily large $k$, in which the constant depends on the strength of the perturbation.
In our experiment, we pick a perturbation at random, and it may be unlikely that we get one whose maximum area and maximum diameter do not tend to zero.
\begin{figure}[hbt]
  \centering
    \vspace{0.0in}
     \includegraphics[width=0.62\textwidth]{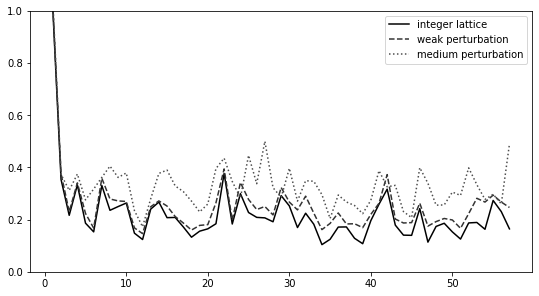}
     \includegraphics[width=0.62\textwidth]{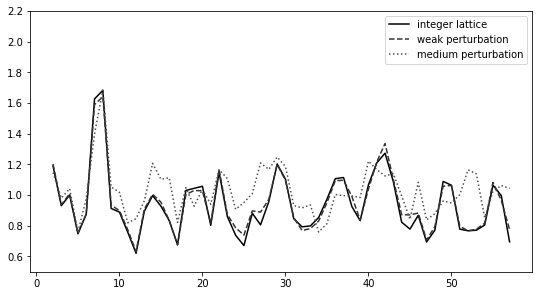}
    \caption{\emph{Top:} the maximum area of a chamber in the $k$-th Brillouin zone of $0$ in $\Zspace^2$ times $k^{{1}/{2}}$, and the same for two perturbations of the integer lattice.
    \emph{Bottom:} the maximum diameter of a chamber in the $k$-th Brillouin zone of $0$ times $k^{{1}/{4}}$ in the same sets.
    Similar to Figure~\ref{fig:width}, the curves of the weak and medium perturbations may be inaccurate beyond $k = 56$ and $k = 52$, respectively.}
  \label{fig:maxareadiameter}
\end{figure}

\section{Discussion}
\label{sec:7}

Brillouin zones originate with L\'{e}on Brillouin's work on the propagation of electron waves in crystal lattices \cite{Bri46}.
This paper addresses fundamental geometric and combinatorial questions about these zones.
The study is interesting already for the integer lattices, which together with their perturbations are the focus of this paper.
We expect that most of our findings extend to general lattices in Euclidean space, if not verbatim then in spirit.
Some of our results also extend to broader categories, such as periodic sets and Delone sets; see the comments following the proofs of the theorems throughout this paper.
The reported work opens a number of questions for future inquiry:
\smallskip \begin{itemize}
  \item[1] Are the bounds proved throughout this paper tight?
    In particular:
    \smallskip \begin{itemize}
      \item[a] Can the $O(k)$ bound on the number of chambers of $0 \in \Zspace^2$ proved in Theorem~\ref{thm:chambers_for_2D_integer_lattice} be complemented with an $\Omega (k)$ bound for the same quantity?
      \item[b] Can the upper bound on the number of chambers in the $k$-th Brillouin zone of $0 \in \Zspace^d$ given in Theorem~\ref{thm:chambers_for_integer_lattices} be improved to $O(k^{d-1})$?
      Equivalently, does the inversion of $\Zspace^d$ through the unit sphere in $\Rspace^d$ have only $O(k^{d-1})$ $k$-sets?
      \item[c] Is the bound on the maximum diameter of a chamber in the $k$-th Brillouin zone of $0 \in \Zspace^d$ proved in Theorem~\ref{thm:size_for_integer_lattices} asymptotically tight?
      Figure~\ref{fig:maxareadiameter} suggests it is in $\Rspace^2$.
    \end{itemize} \smallskip
  \item[2] Recall the notion of distortion of the boundary of the $k$-th Brillouin zone introduced in Section~\ref{sec:4.3} (see Figure~\ref{fig:perimeter}).
  Prove or disprove that the distortion converges to $\sfrac{4}{\pi}$, which we note is the universal constant for the distortion of Voronoi paths in the plane \cite{BTZ00,EdNi21}.
\end{itemize}




\begin{thebibliography}{21}

{\footnotesize

\bibitem{ABFK92}
{\sc N.\ Alon, I.\ B\'{a}r\'{a}ny, Z.\ F\"{u}redi and D.J.\ Kleitman.}
Point selections and weak $\varepsilon$-nets for convex hulls.
\emph{Combinatorics, Probability and Computing} {\bf 1} (1992), 189--200.

\bibitem{BG13}
{\sc M.\ Baake, U.\ Grimm.}
\emph{Aperiodic Order. Volume 1: A Mathematical Invitation.}
Cambridge University Press, Cambridge, UK, 2013.

\bibitem{BTZ00}
{\sc F.\ Baccelli, K.\ Tchoumatchenko and S.\ Zuyev.}
Markov paths on the Poisson--Delaunay graph with applications to routing in mobile networks.
\emph{Adv.\ Appl.\ Probab.} {\bf 32} (2000), 1--18.

\bibitem{Bie39}
{\sc L.\ Bieberbach.}
\"{U}ber die Inhaltsgleichheit der Brillouinschen Zonen.
\emph{Monatshefte f\"{u}r Math.\ Phys.} {\bf 48} (1939), 509--515.

\bibitem{BCES21}
{\sc R.\ Biswas, S.\ Cultrera di Montesano, H.\ Edelsbrunner and M.\ Saghafian.}
Counting cells of order-$k$ Voronoi tessellations in $\Rspace^3$ with Morse theory.
In ``Proc.\ 37th Ann.\ Sympos.\ Comput.\ Geom., 2021'', 16:1--16:15.

\bibitem{Bri30}
{\sc L.\ Brillouin.}.
Les {\'e}lectrons libres dans les m{\'e}taux et le r{\^o}le des r{\'e}flexions de Bragg. 
\emph{J.\ Phys.\ Radium} {\bf 1} (1930), 377-400.

\bibitem{Bri46}
{\sc L.\ Brillouin.}
Wave Propagation in Periodic Structures, Electric Filters and Crystal Lattices.
Dover, 1946.

\bibitem{CoHi07}
{\sc S.\ Cooper and M.\ Hirschhorn.}
On the number of primitive representations of integers as sums of squares.
\emph{Ramanujan J.} {\bf 13} (2007), 7--25.

\bibitem{Corput20}
{\sc J.G.\ van der Corput.}
{\"U}ber Gitterpunkte in der Ebene.
\emph{Math. Annal.} {\bf 81} (1920), 1--20.

\bibitem{Dey98}
{\sc T.K.\ Dey.}
Improved bounds on planar $k$-sets and related problems.
\emph{Discrete Comput.\ Geom.} {\bf 19} (1998), 373--382.

\bibitem{Ede87}
{\sc H.\ Edelsbrunner.}
\emph{Algorithms in Combinatorial Geometry.}
Springer, Heidelberg, Germany, 1987.

\bibitem{EHKSW21}
{\sc H.\ Edelsbrunner, T.\ Heiss, V.\ Kurlin, P.\ Smith and M.\ Wintraecken.}
The density fingerprint of a periodic point set.
\emph{In} ``Proc.\ 37th Ann.\ Sympos.\ Comput.\ Geom., 2021'', 32:1--32:16.

\bibitem{EdIg18}
{\sc H.\ Edelsbrunner and M.\ Iglesias-Ham.}
On the optimality of the FCC lattice for soft sphere packing.
\emph{SIAM J.\ Discrete Math.} {\bf 32} (2018), 750--782.

\bibitem{EdNi21}
{\sc H.\ Edelsbrunner and A.\ Nikitenko.}
Average and expected distortion of Voronoi paths and scapes.
\texttt{arXiv:2012:03350v2[math.MG]}, 2021.

\bibitem{EdSe86}
{\sc H.\ Edelsbrunner and R.\ Seidel.}
Voronoi diagrams and arrangements.
\emph{Discrete Comput.\ Geom.} {\bf 1} (1986) 25--44.

\bibitem{EMS04}
{\sc P.\ Engel, L.\ Michel and M.\ Senechal.}
\emph{Lattice Geometry.}
Preprint, \texttt{http://www.ihes.fr}, 2004.

\bibitem{Gha23}
{\sc Mohadese Ghafari.}
Supplemental code for \emph{Brillouin Zones of Integer Lattices and Their Perturbations},
\url{https://github.com/almaho/Brillouin_Zones_Integer_Lattices_Perturbations/}, 2023.

\bibitem{Gru07}
{\sc P.M.\ Gruber.}
\emph{Convex and Discrete Geometry.}
Volume 336 of Grundlehren Math.\ Wiss., Springer, Berlin, 2007.

\bibitem{HaWr08}
{\sc G.H.\ Hardy and E.M.\ Wright.}
\emph{An Introduction to the Theory of Numbers}.
Sixth edition, Oxford Univ.\ Press, Oxford, England, 2008.

\bibitem{Hux03}
{\sc M.\ Huxley.}
Exponential sums and lattice points III.
\emph{Proc.\ London Math.\ Soc.} {\bf 87} (2003), 591--609.

\bibitem{Jones03}
{\sc G.A.\ Jones.}
Geometric and asymptotic properties of Brillouin zones in lattices.
\emph{Bull.\ London Math.\ Soc.} {\bf 16.3} (1984), 241--263.

\bibitem{kwakkel06}
{\sc F.H.\ Kwakkel.}
Rigidity of Brillouin zones.
Master thesis, Department of Mathematics, Rijksuniversiteit Groningen, the Netherlands, 2006.

\bibitem{Lee82}
{\sc D.-T.\ Lee.}
On $k$-nearest neighbor Voronoi diagrams in the plane.
\emph{IEEE Trans.\ Comput.} {\bf 31} (1982), 478--487.

\bibitem{Ree11}
{\sc D.\ Reem.}
The geometric stability of Voronoi diagrams with respect to small changes of the sites.
\emph{In} ``Proc.\ 27th Ann.\ Sympos.\ Comput.\ Geom., 2011'', 254--263.

\bibitem{ScWe08}
{\sc R.\ Schneider and W.\ Weil.}
\emph{Stochastic and Integral Geometry.}
Springer-Verlag, Berlin, Germany, 2008.

\bibitem{SST01}
{\sc M.\ Sharir, S.\ Smorodinsky and G.\ Tardos.}
An improved bound for $k$-sets in three dimensions.
\emph{Discrete Comput.\ Geom.} {\bf 26} (2001), 195--204.

\bibitem{Skr87}
{\sc M.M.\ Skriganov.}
Brillouin zones and the geometry of numbers.
\emph{J.\ Soviet Math.} {\bf 36} (1987), 140--154.

\bibitem{Tor18}
{\sc S.\ Torquato.}
Hyperuniform states of matter.
\emph{Physics Reports} {\bf 745} (2018), 1--95.

\bibitem{Voi08}
{\sc I.K.\ Voigt.}
\emph{Voronoizellen diskreter Punktmengen.} 
PhD Thesis, Dept.\ Math., Techn.\ Univ.\ Dort\-mund, Germany, 2008. 

\bibitem{Zhi15}
{\sc B.\ Zhilinskii.}
\emph{Introduction to Louis Michel's Lattice Geometry through Group Action.}
EDP Sciences, ENRS Editions, Paris, France, 2015.

}
\end{thebibliography}

\clearpage
\appendix

\section{Computational Background}
\label{app:A}

We support our theoretical findings with data that counts and measures Brillouin zones and their chambers in $\Rspace^2$.
To generate this data for the integer lattice, we use a cut-off,
$m > 0$, and construct the arrangement defined by the bisectors of $0$ and all integer points in $[-m,m]^2 \setminus \{0\}$, denoted $\Acal_m$.
There are about $4 m^2$ lines forming an arrangement of $O(m^4)$ vertices, edges, and chambers.
Using a classic incremental algorithm, $\Acal_m$ can be computed in time $O(m^4)$ \cite[Chapter 7]{Ede87}.
We use $\Acal_m$ to collect the data, and for this we need to know how many Brillouin zones in $\Acal_m$ are also in $\Acal_\infty$.

\smallskip
Similarly, we collect the data for perturbations of $\Zspace^2$ by constructing the corresponding arrangements of perturbed bisectors.
As explained in Section~\ref{sec:2.5}, we require that $0 \in \Zspace^2$ is not perturbed, and we call the perturbation \emph{weak}, \emph{medium}, and \emph{strong} if every integer point, $a \in \Zspace^2$, is mapped uniformly at random to $\varphi (a) \in a + [-\tau, \tau]^2$, in which $\tau = 0.02$, $0.10$, and $0.50$, respectively.
Write $\Acal_m (\varphi)$ for the arrangement of bisectors defined by $0$ and all $\varphi (a)$ with $a \in [-m,m]^2 \setminus \{0\}$.
Again we need to know how many Brillouin zones of $\Acal_m (\varphi)$ are also in $\Acal_\infty (\varphi)$.

\begin{lemma}[\#Correct Brillouin Zones]
  \label{lem:correct_Brillouin_zones}
  Let $\varphi \colon \Zspace^2 \to \Rspace^2$ be a perturbation of strength $\tau$.
  For every $k < \frac{\pi}{4} [m+1-\sqrt{2}- (2\sqrt{2}+1) \tau]^2$, the $k$-th Brillouin zone of $0$ in $\Acal_m (\varphi)$ is also the $k$-th Brillouin zone of $0$ in $\Acal_\infty (\varphi)$.
\end{lemma}
\begin{proof}
  Let $b \in \Zspace^2 \setminus [-m,m]^2$ and note that the distance of $0$ from the bisector defined by $0$ and $\varphi (b)$ is at least $\frac{1}{2}[m+1-\tau]$.
  The restriction of $\Acal_m (\varphi)$ to the disk with center $0$ and radius $\frac{1}{2} [m+1-\tau]$ is therefore the same as the restriction of $\Acal_\infty (\varphi)$ to the disk.
  By Theorem~\ref{thm:width_for_perturbed_integer_lattices}, the first $k$ Brillouin zones of $0$ for a perturbation with magnitude $\sqrt{2} \tau$ are contained in the disk of radius $R_k(0) < \sqrt{\sfrac{k}{\pi}} + \sfrac{\sqrt{2}}{2} + \sqrt{2} \tau$.
  Combining these bounds gives $\sqrt{\sfrac{k}{\pi}} + \sfrac{\sqrt{2}}{2} + \sqrt{2} \tau < \frac{1}{2} [m+1-\tau]$, and therefore $k < \tfrac{\pi}{4} [m+1-\sqrt{2}-(2\sqrt{2}+1) \tau]^2$ is a sufficient condition for $R_k(0) < \frac{1}{2}[m+1-\tau]$.
\end{proof}
For example, setting $m = 9$ and $\tau = 0.02$, $0.1$, and $0.5$, we get $k < 56.8$, $52.8$, and $34.9$, respectively.
In words, the data we compute for weak, medium, and strong perturbations are correct up to $k = 56$, $52$, and $34$, respectively, and possibly contaminated by missing bisectors beyond these values of $k$;
see the captions of Figures~\ref{fig:width}, \ref{fig:area}, \ref{fig:perimeter}, \ref{fig:chambers}, and \ref{fig:maxareadiameter}.

\smallskip
Two or more lines in $\Rspace^2$ are \emph{concurrent} if they all share a common point.
A non-trivial numerical aspect of the computations is the recognition of concurrent bisectors as such.
This is especially important when we count chambers in the highly degenerate bisector arrangement defined by the integer lattice.
We use rational arithmetic to recognize concurrency, as we now describe.
Let $a = (a_1,a_2)$ and $b = (b_1,b_2)$ be two linearly independent integer points.
The bisectors of $0, a$ and $0, b$ intersect in the center of the unique circle that passes through $0,a,b$.
To compute this point, we solve the linear system
\begin{align}
    2 a_1 x_1 + 2 a_2 x_2  &=  a_1^2 + a_2^2 , \\
    2 b_1 x_1 + 2 b_2 x_2  &=  b_1^2 + b_2^2 .
\end{align}
Using Cramer's rule, we get $x_1 = {\Delta_1}/{\Delta}$ and $x_2 = {\Delta_2}/{\Delta}$, in which $\Delta = 2 a_1 b_2 - 2 b_1 a_2$, $\Delta_1 = (a_1^2+a_2^2)b_2 - (b_1^2+b_2^2)a_1$, and $\Delta_2 = a_1(b_1^2+b_2^2) - b_1(a_1^2+a_2^2)$.
Assuming the coordinates are integers of absolute size at most $m$, we have $|\Delta|, |\Delta_1|, |\Delta_2| \leq 4 m^3$.
In this paper, we use $m \leq 9$, so each of these integers can be represented by $1 + \log_2 (4 m^3) \leq 13$ bits.

\smallskip
To implement a perturbation, we scale up the integer lattice by a factor $p = 10000$ and randomly pick $\varphi (a)$ from the integer points $p a + [-q,q]^2$, in which $q = 200, 1000, 5000$ depending on whether we desire a \emph{weak}, \emph{medium}, or \emph{strong} perturbation.
This way the coordinates of the perturbed points are still integers so that concurrent lines can be recognized with rational arithmetic as described above.
By construction, the coordinates of the perturbed point have absolute size at most $pm+q$.
Accordingly, the relevant determinants satisfy $|\Delta|, |\Delta_1|, |\Delta_2| \leq 4[pm+q]^3$.
For the above choices, we have $pm+q \leq 95000 < 2^{17}$, so each of them can be represented by $1 + \log_2 (4 [pm+q]^3) \leq 54$ bits.
This is still well within the limit of a $64$-bit computer word.

\end{document}